\documentclass[mnsc,nonblindrev]{informs3_nojournal}

\usepackage{amsfonts,amssymb,mathrsfs,bm,graphicx,booktabs,caption,subcaption,placeins}
\usepackage[colorlinks,citecolor=black,linkcolor=black,urlcolor=black,bookmarks=false,hypertexnames=false]{hyperref}
\usepackage{natbib}
\usepackage{booktabs}
\usepackage{threeparttable}
\usepackage{multirow}
\bibpunct[, ]{(}{)}{,}{a}{}{,}
\def\bibfont{\small}
\def\bibsep{\smallskipamount}
\def\bibhang{24pt}

\OneAndAHalfSpacedXI
\TheoremsNumberedThrough
\EquationsNumberedThrough
\MANUSCRIPTNO{}

\newcommand{\A}{\mathcal A}
\newcommand{\R}{\mathcal R}

\newcommand{\sta}{\mathrm{sta}}

\begin{document}

\RUNAUTHOR{Elmachtoub and Shi}
\RUNTITLE{Tight Guarantees for Static Policies in Vehicle Sharing Networks}
\TITLE{Tight Guarantees for Static Policies in Vehicle Sharing Networks}

\ARTICLEAUTHORS{
\AUTHOR{Adam N. Elmachtoub}
\AFF{Department of Industrial Engineering and Operations Research and Data Science Institute, Columbia University, New York, NY 10027, \EMAIL{adam@ieor.columbia.edu}}
\AUTHOR{Jiaqi Shi}
\AFF{Department of Industrial Engineering and Operations Research and Data Science Institute, Columbia University, New York, NY 10027, \EMAIL{js5778@columbia.edu}}
}

\ABSTRACT{We study pricing in a vehicle sharing network with $N$ locations and $C$ vehicles, motivated by on-demand bike, scooter, and car sharing, as well as robotaxi and ride sharing services.  Customers request directed trips between locations, with travel times following general distributions, and leave without service if no vehicle is available. Pricing affects both demand for trips and the distribution of vehicles across locations. Computing an optimal dynamic pricing policy is challenging because the system state depends on the number of vehicles available at each location and in transit (including how long they have been traveling), motivating the use of simple static policies that offer a fixed price for each trip.

For any number of vehicles and locations, we establish a tight performance guarantee for static pricing relative to optimal dynamic pricing. The guarantee holds for any concave objective functions (including  revenue rate) and general travel time distributions. Our tight guarantee is bounded below by $ \frac{C}{C + \min \{ C -1, \sqrt{ CN} \} }$ (over 50\%), achieving the first universal, non-asymptotic constant factor guarantee. We also prove that the optimal static policy can be efficiently computed by solving a convex optimization problem. Finally, we extend our performance guarantee to the setting with supply redirection.

}


\maketitle

\section{Introduction}\label{sec:introduction}
Vehicle sharing networks are prevalent in transportation systems, from bike sharing (Citi Bike), scooter sharing (Lime), and car sharing (Zipcar) to ride sharing (Uber, Lyft, and Didi) and emerging robotaxi services (Waymo and Pony.ai). Platforms operating these on-demand services face a fundamental challenge in matching supply with demand across locations. Differences in demand, together with randomness in trip requests and travel times, can leave some locations short of vehicles while others have idle vehicles. Dynamic pricing can help address this mismatch by adjusting trip prices in response to vehicle availability. However, pricing decisions across different locations are interdependent: changing demand for trips from one location affects future vehicle availability elsewhere in the network. Computing an optimal dynamic pricing policy requires accounting for these interactions over a large state space describing the number of vehicles at different locations and the amount of travel time incurred in ongoing trips, which is intractable due to the curse of dimensionality. This computational difficulty motivates a simple static pricing policy that offers a fixed price for each origin–destination trip at all times.


Previous work has demonstrated the effectiveness of static pricing relative to optimal dynamic pricing. For a network with $C$ vehicles and $N$ locations, \cite{waserhole2016pricing} establish a $ \frac{C}{C + N -1} $ guarantee for static pricing when travel times are zero. \cite{banerjee2022pricing} extend the analysis to nonzero travel times and obtain the guarantee $\left(1 - 2\sqrt{\frac{\log C}{C}}\right) \left(\frac{\sqrt{C\log C}}{\sqrt{C\log C} + N-1} - \frac{3}{\sqrt{C\log C}}\right) $ under additional conditions (including  $ C \geq 100 $). 
\cite{benjaafar2023pricing} subsequently obtain an improved guarantee of $ 1 + \frac{N-1}{2C}+\frac{1}{2(N+C)}-\sqrt{\left(\frac{N-1}{2C} + \frac{1}{2(N+C)}\right)^2 + \frac{N-1}{C}+\frac{1}{N+C}} $ for any $C$ and $N$. These results show that static pricing is powerful when the number of vehicles is large relative to the number of locations. In particular, all these guarantees approach one as $C$ grows for fixed $N$. 

However, these guarantees may be near zero when $N$  substantially exceeds $C$, leaving limited theoretical support for static pricing in practical settings such as new market launches, low-density suburban markets, and late nights with fewer drivers. For example, in emerging robotaxi services, Pony.ai reported over 10,000 pickup and drop-off locations in Shenzhen in November 2025, while its entire Robotaxi fleet comprised 961 vehicles \citep{pony2025fleet}. In ride sharing, the number of geographic zones (locations) can exceed the number of participating vehicles in low-density suburban markets, where Uber reports longer waits and pickup distances \citep{uber2024remarks}. During late-night operations, vehicle supply can decline substantially, as documented by SFCTA for Uber and Lyft trips \citep{sfcta2017tncstoday}. These settings motivate performance guarantees of static pricing that remain meaningful when $ N > C$. 

To this end, for any $C$ and $N$, we establish the tight $ \gamma_{C,N}:=\frac{\sum_{k=0}^{C-1} \frac{1}{k!} (C-\min\{N,C\})^k  \binom{\min\{N,C\} + C-k-2}{C-k-1}}{\sum_{k=0}^{C} \frac{1}{k!} (C-\min\{N,C\})^k  \binom{\min\{N,C\}+C-k-1}{C-k}} $ performance guarantee for static pricing compared to the optimal dynamic pricing (Theorem \ref{thm:main-guarantee}). This guarantee holds for any concave objective function (including revenue rate) and general travel time distributions. Our tight guarantee satisfies $\gamma_{C,N} \geq  \frac{C}{C + \min \{ C -1, \sqrt{ CN} \} } \geq \frac{1}{2}$, providing the first universal, non-asymptotic 50\% guarantee. With zero travel times, we establish the tight guarantee $ \frac{C}{C+\min\{C,N\}-1} $ (Theorem \ref{thm:zero-travel}). Our analysis does not rely on the flow-relaxation method used in \cite{waserhole2016pricing, banerjee2022pricing} and \cite{benjaafar2023pricing}, which imposes additional network balance constraints and equal availability of vehicles at each node. Rather, our approach directly constructs a static policy from the optimal dynamic pricing policy, which we ultimately preserves the same relative availability across locations. Allowing availability to vary across locations is valuable because some trips generate greater revenue, making it more beneficial to keep vehicles available where these higher-value trips begin. We also show that the optimal static policy can be efficiently computed by solving one convex optimization problem (Theorem \ref{thm:convex-static-program}). Finally, we extend our $ \gamma_{C,N} $  guarantee for static policies to settings with supply redirection (Theorem \ref{thm:supply-redirection}). Together, these results show both the efficacy and the computational tractability of static policies for vehicle sharing networks.

\subsection{Related Literature}\label{sec:literature} 
Beyond the static pricing guarantees in vehicle sharing networks discussed above, related research examines dynamic pricing, strategic drivers, and supply redirection. \cite{balseiro2021dynamic} develop a Lagrangian-based dynamic pricing policy for hub-and-spoke networks, establishing asymptotic optimality as the number of spokes and vehicles grow proportionally. \cite{chen2024real} consider a finite horizon with nonstationary demand and deterministic travel times, developing a dynamic pricing policy that corrects accumulated demand fluctuations and improves the asymptotic revenue-loss rate relative to state-independent pricing. Other approaches to queueing network control include internal pricing for admission decisions in closed logistics networks \citep{adelman2007price} and mirror backpressure policies that use queue lengths to manage supply under unknown, potentially time-varying demand rates \citep{kanoria2024blind}. Research on strategic drivers examines pricing and compensation with endogenous joining and participation \citep{cachon2017role}, spatial price discrimination with endogenous entry and relocation \citep{bimpikis2019spatial}, and spatial pricing with costly driver relocation and fixed total supply \citep{besbes2021surge}. \cite{afeche2023ride} examine admission and repositioning control in a two-location fluid model with strategic drivers. For supply redirection and repositioning, \cite{braverman2019empty} establish the asymptotic optimality of a fluid-based static randomized routing policy as demand and fleet size grow, \cite{benjaafar2022dynamic} characterize a no-repositioning region and develop a convergent approximate dynamic programming algorithm for periodic repositioning with continuous inventory, and \cite{hosseini2025dynamic} develop a fluid-based dynamic relocation policy for car-sharing networks.

Our work also relates to pricing for reusable resources in Erlang loss systems, with the one-class customer corresponding to the special case $N=1$ of our model. \cite{besbes2022static} show that static pricing achieves at least $78.9\%$ of the optimal dynamic pricing revenue under regular valuation distributions and exponential service times. \cite{elmachtoub2023power} extend the $78.9\%$ guarantee to multiple customer classes and general service time distributions, and establish a $97.5\%$ guarantee relative to the optimal inventory-based policy for one customer class with a monotone hazard rate valuation distribution. \cite{balseiro2026dynamic} show that policies using only two prices can achieve faster convergence to the fluid upper bound than static pricing as the number of units grows. Other studies investigate asymptotically optimal pricing, admission control, and pricing with advance reservations \citep{Paschalidis2000, levi2010provably, lei2020real}. 


\section{Model and Preliminaries}\label{sec:model}
We consider a network with $N$ locations, indexed by $[N]=\{1,\ldots,N\}$, and a directed arc set $\A\subseteq[N] \times[N]$. The locations can represent geographic zones in ride hailing such as Uber, docking stations in bike sharing such as Citi Bike, rental locations in car rental systems, or pick-up and drop-off locations in robotaxi platforms. Each arc $(i,j)\in\A$ represents a feasible customer trip from location $i$ (origin) to location $j$ (destination). We allow $ (i,i) \in \A $, representing self-loop trips. There are $C$ vehicles moving on this network. Potential customers requesting trip $(i,j)\in\A$ arrive according to an independent Poisson process with rate $\Lambda_{ij}$. Each $(i,j)$ customer has an i.i.d. valuation for the trip drawn from a known distribution $F_{ij}$. Upon arrival, we call an $(i,j)$ customer willing if their valuation exceeds the current price $ p^{ij} $.
The travel time of the trip from location $i$ to location $j$ is drawn i.i.d. from a general nonnegative distribution $H_{ij}$ with finite mean $ 1/\mu^{ij} $. We assume travel times are independent across trips and independent of arrivals and valuations. After completing a trip, the vehicle becomes idle at the destination. Due to the on-demand needs, customers who arrive when no vehicle is available at their origin leave immediately without service.

As is common with pricing, it is easier to work in the quantile space. We let $ \lambda^{ij}: = \Lambda_{ij} ( 1 - F_{ij}(p^{ij}) )  $ denote the effective arrival rate, i.e., the arrival rate of willing $(i,j)$ customers under price $ p^{ij} $. We assume that $F_{ij}$ is supported on $ [ \underline{v}_{ij}, \bar{v}_{ij} ]$ where $\bar{v}_{ij} $ can be infinity, and has a continuous positive density function $f_{ij} $ on $ ( \underline{v}_{ij}, \bar{v}_{ij} ) $. Then, for any $ \lambda^{ij} \in (0, \Lambda^{ij} ) $, there is a one-to-one correspondence between effective arrival rates and prices, denoted by $ p^{ij}(\lambda^{ij}) $. We further define $p^{ij}(0) = \bar{v}_{ij} $ and $ p^{ij}( \Lambda_{ij} ) = \underline{v}_{ij} $. Let $r^{ij} (\lambda) = \lambda \cdot p^{ij}(\lambda)  $ denote the revenue rate function. When $ \bar{v}_{ij} = \infty $, we work with the continuous closure of the revenue rate function and define $r^{ij}(0):=\lim_{\lambda \to 0} r^{ij}(\lambda) $. For any $(i,j) \in \A $, we assume that $r^{ij}(\cdot)$ is concave, a standard assumption in the literature. Given a pricing policy, the above system can be viewed as a closed queueing network from vehicles' perspective with two types of queues: (i) a single-server FCFS queue at each location whose exponential service completion is the arrival of a willing customer (ii) an infinite-server queue (transit queue) for each arc with general service times ($H_{ij}$) that represents vehicles carrying customers from the origin to the destination.

Let $x_i$ denote the number of idle vehicles at location $i$ and $x_{ij}$ be the number of vehicles traveling on arc $(i,j)$. We use $  \mathcal X_{C,N} :=\left\{(x_i:i\in[N],x_{ij}:(i,j)\in\A)\in\mathbb Z_+^{N+|\A|}: \sum_{i=1}^N x_i + \sum_{(i,j) \in\A} x_{ij} = C\right\} $ to denote the count space. As we consider general travel times without the memoryless property in this paper, $\mathcal{X}_{C,N}$ itself is not exactly the system state space. To make the system Markovian, the system space needs to further track the travel time spent so far for every busy vehicle. Specifically, given $x_{ij}$, we define $ \bm{y}^{ij} = (y_1^{ij},\ldots,y_{x_{ij}}^{ij} ) $ with $ y_k^{ij} $ representing the travel time spent so far of vehicle moving from $i$ to $j$ with label $k$. The system space is then defined as
\begin{align*}
    \mathcal S_{C,N} = 
    \bigcup_{\bm x \in \mathcal{X}_{C,N} } \left\{ \left(\bm x, \bm y^{ij}, (i,j) \in \A \right) : \bm{y}^{ij} \in \mathbb{R}_+^{x_{ij}}, \ \forall (i,j)\in \A   \right\} =: \bigcup_{\bm x \in \mathcal{X}_{C,N} }
    \mathcal S_{C,N}(\bm x).
\end{align*}
A dynamic pricing policy offers price $p^{ij}_{\bm s}$ when an $(i,j)$ customer arrives to the system in state $\bm s \in \mathcal{S}_{C,N}$, which induces an effective arrival rate $ \lambda^{ij}_{\bm s} $.  A static pricing policy offers price $p^{ij}$ independent of the state for all $ (i,j) $ customers, which induces an effective arrival rate $ \lambda^{ij} $. Throughout the paper, we will call $\lambda$ the effective arrival rate and pricing policy interchangeably.

Let $ \bm \lambda =  ( \lambda^{ij}_{\bm s} )_{ (i,j) \in \A, \bm s \in \mathcal{S}_{C,N} } $ denote some arbitrary pricing policy and $\Pi^{\bm \lambda} $ be its corresponding stationary probability measure over $ \mathcal{S}_{C,N}$. For any $ \bm x \in \mathcal{X}_{C,N}  $, we denote the steady-state probability of being in count state $\bm x$ by $ \mathbb{P}_{\bm x} (\bm \lambda)  = \Pi^{\bm \lambda } ( \mathcal{S}_{C,N}(\bm x) ) $. Our performance metric is the long-run average revenue rate (our results hold for any concave performance metrics such as throughput, service level, and social welfare noted in \citealp{banerjee2022pricing}), defined as
\begin{align}\label{eq:dynamic-revenue}
\mathcal{R}(\bm\lambda)=\sum_{(i,j)\in \A}\int_{\mathcal S_{C,N}}  \mathbf 1\{x_i\geq1\} \; r^{ij}(\lambda_{\bm s}^{ij}) \ d \Pi^{\bm\lambda}( \bm{s}).
\end{align}
For each $(i,j)\in\mathcal A$, we define the throughput (the long-run average rate at which willing $(i,j)$ customers are served) as $v_{ij}(\bm \lambda):=\int_{\mathcal S_{C,N}}\mathbf 1\{x_i\geq 1\}\lambda_{\bm s}^{ij} d \Pi^{\bm \lambda}( \bm s)$. 
We denote by $\bm\lambda^*$ and  $\bm\lambda^{sta*}$ an optimal dynamic and static policy, respectively, and let $\mathcal{R}^*$ and $ \mathcal{R}^{sta*} $ be their corresponding long-run average revenue rates. 
We assume that the system induced by $\bm\lambda^*$ is ergodic.  We define the induced network as the network consisting of arcs with $v_{ij}(\bm\lambda^*)>0$ and assume that it is strongly connected.

Obviously, computing the optimal dynamic pricing policy is challenging since the state space $\mathcal{S}_{C,N}$ is infinite. This motivates us to consider static pricing that offers a fixed price for each trip $ (i,j) \in \A $. Let $ \Omega_{C,N}: = \{ (C,N, \Lambda_{ij}, H_{ij}, r_{ij}(\cdot), (i,j) \in \A ): \Lambda_{ij} > 0, H_{ij} \text{ is nonnegative} , r^{ij}(\cdot) \text{ is concave}   \} $ denote the family of instances satisfying the above assumptions. Our goal in this work is to establish the universal guarantee of static pricing compared to the optimal dynamic pricing policy over the instances in  $ \Omega_{C,N} $, i.e.,
\begin{align*}
  \inf_{\Omega_{C,N}} \frac{\mathcal{R}^{sta*}}{\mathcal{R}^{*}}
\end{align*}
When $C=1$, we note that static and dynamic pricing are equivalent because trips are accepted at each location only when the vehicle is idle at that location. Thus $\R^{sta*}=\R^{*}$ for $C=1$. We therefore focus on the nontrivial case $C\geq2$ in the remainder of the paper.


\subsection{Product-form Queueing Network under a Static Policy}\label{sec:queueing-preliminaries}


For a static pricing policy $\bm \lambda = (\lambda^{ij})_{(i,j) \in \A }$ that induces a strongly connected network, let $\lambda_i = \sum_{j:(i,j) \in \A} \lambda^{ij} $ denote the total effective arrival rate of trips whose origin is location $i$. Under the static policy, the induced network belongs to the classical closed BCMP network \citep{baskett1975open} whose steady-state probabilities admit the celebrated product-form expressions. 
Let $ w_i $, $i \in [N]$, be a solution to the following traffic equations of the closed network:
\begin{align}\label{eq: traffic}
    w_i  = \sum_{k : (k,i) \in \A } w_k \frac{\lambda^{ki}}{ \lambda_k }, \quad i \in [N].
\end{align}
Then, as shown in \cite{baskett1975open}, the steady-state probability $ \mathbb P_{\bm x}(\bm\lambda) $ under static policy $\bm \lambda = (\lambda^{ij})_{(i,j) \in \A }$ is given by
\begin{align}\label{eq:product-form}
\mathbb P_{\bm x}(\bm\lambda) = \frac{\prod_{i=1}^N \left(\frac{w_i }{\lambda_i} \right)^{x_i} \prod_{(i,j)\in\A} \frac{1}{x_{ij} ! } \left(\frac{w_i \lambda^{ij} }{\lambda_i \mu^{ij}}\right)^{x_{ij}}  }{\sum_{\bm x' \in \mathcal X_{C,N}}\prod_{i=1}^N \left(\frac{w_i }{\lambda_i} \right)^{x'_i} \prod_{(i,j)\in\A} \frac{1}{x'_{ij} !} \left(\frac{w_i \lambda^{ij} }{\lambda_i \mu^{ij}}\right)^{x'_{ij}} }
,\qquad \bm x\in\mathcal X_{C,N}.
\end{align}

The above expression holds for general travel times by the insensitivity property \citep{barbour1976networks,whitt1980continuity,taylor2010insensitivity}, under which the steady-state probability depends on the travel time distributions only through their means ($1/\mu^{ij}$).

 

\section{The Performance of Static Pricing}\label{sec:main-result}


Our approach is to choose a static pricing policy instead of focusing on the optimal static pricing policy. Unlike previous work that uses a flow-based policy obtained by solving a convex optimization problem, we directly construct a static policy from the optimal dynamic policy. 
Let $ \alpha_i( \bm \lambda^* ):=\sum_{\bm x \in\mathcal X_{C,N}:x_i\geq1}\mathbb P_{\bm x}(\bm\lambda^*) = \int_{\mathcal S_{C,N}} \mathbf 1\{x_i\geq 1\} \; d \Pi^{\bm\lambda^*} \! ( \bm s )   $ denote the availability (service level) at location $i$ and $v_{ij}(\bm\lambda^*)=\int_{\mathcal S_{C,N}} \mathbf 1\{x_i\geq 1\}\lambda_{\bm s}^{ij*} \; d \Pi^{\bm\lambda^*} \! ( \bm s )$ be the throughput under the optimal dynamic policy $ \bm{\lambda}^* $. For each $(i,j) \in \A$, our constructed static policy, $\tilde\lambda^{ij}$, is defined as the expected effective arrival rates conditional on at least one vehicle being available at location $i$, i.e.,
\begin{align}\label{eq:constructed-policy}
\tilde\lambda^{ij} =\frac{\int_{\mathcal S_{C,N}} \mathbf 1\{x_i\geq 1\}\lambda_{\bm s}^{ij*} \; d \Pi^{\bm\lambda^*}\! (\bm s)}{\sum_{\bm x\in \mathcal X_{C,N}:x_i\geq1}\mathbb P_{\bm x}(\bm\lambda^*)} = 
\frac{v_{ij}(\bm\lambda^*)}{\alpha_i(\bm\lambda^*)}.
\end{align}
Since we assume the network induced by $ \bm \lambda^* $ is strongly connected, $ \alpha_i(\bm \lambda^*)>0 $, which implies that the network induced by the constructed static policy is also strongly connected. The construction of our static policy exploits the fact that revenue is only generated at a location when vehicles are available, which is a key characteristic of a loss system. Let $ \tilde{\bm \lambda} = (\tilde{\lambda}^{ij})_{(i,j) \in \A} $ denote this static policy and let $\R^{\tilde{\bm \lambda}}$ be its long-run average revenue rate, and $ \alpha_i( \tilde{\bm \lambda}) $ be the corresponding availability at location $ i \in [N] $. A key property of the constructed static policy is that the availability ratios $\alpha_i(\tilde{\bm\lambda})/\alpha_i(\bm\lambda^*)$ are the same across all locations. The following lemma shows that we can use $\tilde{\bm\lambda}$ to lower bound the revenue ratio $\R^{sta*}/\R^*$ by this common availability ratio.

\begin{lemma}\label{lem: RevRatio_lower}
    The constructed static policy satisfies
\begin{align}\label{eq: rev-ratio-bound}
\frac{\R^{sta*}}{\R^*} \geq 
\frac{\R^{\tilde{\bm\lambda}}}{\R^{*}}
\geq  \frac{ \sum_{(i,j)\in\A} \alpha_i(\bm{\tilde \lambda} ) r^{ij}( \tilde{\lambda}^{ij})  }{\sum_{(i,j)\in\A}\alpha_i(\bm\lambda^*)r^{ij}(\tilde\lambda^{ij}) }
= \frac{\sum_{\bm x \in \mathcal X_{C-1,N}}\prod_{i=1}^N \left( \alpha_i (\bm \lambda^*)  \right)^{x_i} \prod_{(i,j)\in\A} \frac{1}{x_{ij}! }  \left(\frac{v_{ij}(\bm{\lambda}^*) }{\mu^{ij}} \right)^{x_{ij}} }{\sum_{\bm x \in \mathcal X_{C,N}}\prod_{i=1}^N \left( \alpha_i (\bm \lambda^*)  \right)^{x_i} \prod_{(i,j)\in\A}\frac{1}{x_{ij}! }  \left(\frac{v_{ij}(\bm{\lambda}^*) }{\mu^{ij}} \right)^{x_{ij}} } .
\end{align}
\end{lemma}
\begin{proof}{Proof.}
According to the definitions of $ \R^{\tilde{\bm\lambda}} $ and $ \R^{*}$,
\begin{align}\label{eq:lower-rev-ratio }
  \frac{\R^{\tilde{\bm\lambda}}}{\R^{*}} = \frac{ \sum_{(i,j)\in\A} \alpha_i(\bm{\tilde \lambda} ) r^{ij}( \tilde{\lambda}^{ij})  }{\sum_{(i,j)\in\A} \int_{\mathcal S_{C,N}}\mathbf 1\{x_i\geq1\} r^{ij}(\lambda_{\bm s}^{ij*}) \; d \Pi^{\bm\lambda^*} \! ( \bm{s})} \notag
  \geq& \frac{ \sum_{(i,j)\in\A} \alpha_i(\bm{\tilde \lambda} ) r^{ij}( \tilde{\lambda}^{ij})  }{ \sum_{(i,j)\in\A} \alpha_i(\bm\lambda^*) r^{ij}\left(\frac{\int_{\mathcal S_{C,N}}\mathbf 1\{x_i\geq1\} \lambda_{\bm s}^{ij*} \; d \Pi^{\bm{\lambda}^*}\! (\bm{s})}{\alpha_i(\bm\lambda^*)} \right) } \notag \\
  =&  \frac{ \sum_{(i,j)\in\A} \alpha_i(\bm{\tilde \lambda} ) r^{ij}( \tilde{\lambda}^{ij})  }{\sum_{(i,j)\in\A}\alpha_i(\bm\lambda^*)r^{ij}(\tilde\lambda^{ij}) },
\end{align}
where the inequality follows from Jensen's inequality applied to the concave function $r^{ij}(\cdot)$ and the second equality follows from our constructed static policy \eqref{eq:constructed-policy}.  

Next, we observe a core property that there is a common ratio $\alpha$ such that $ \frac{\alpha_i(\tilde{\bm\lambda})}{\alpha_i(\bm\lambda^*)} = \alpha $ across all locations. Note that in steady state, the rate at which vehicles flow into a location equals the rate at which they flow out of that location under the optimal dynamic policy,
\begin{align*}
v_i (\bm\lambda^*): =  \sum_{j:(i,j)\in\A}v_{ij}(\bm\lambda^*) = \sum_{k:(k,i)\in\A} v_{ki}(\bm\lambda^*), \quad i\in[N].
\end{align*}
Let $ \tilde{\lambda}_i = \sum_{j:(i,j) \in \A} \tilde{\lambda}^{ij}  $. From the derived equality above and the fact that $ v_{ij} (\bm \lambda^*) = \tilde{\lambda}^{ij} \alpha_i (\bm{\lambda}^* ) $ from our construction \eqref{eq:constructed-policy}, we have that
\begin{align*}
\tilde{\lambda}_i \alpha_i (\bm \lambda^*) = \sum_{k:(k,i)\in\A}  \tilde{\lambda}^{ki} \alpha_k(\bm \lambda^*) 
= \sum_{k:(k,i)\in\A} \tilde{\lambda}_k \alpha_k(\bm \lambda^*) \frac{\tilde{\lambda}^{ki} }{\tilde{\lambda}_k}   , \quad i\in[N].
\end{align*}
Thus, $ \tilde{\lambda}_i \alpha_i (\bm \lambda^*) $, $i \in [N] $, is a solution to the traffic equations \eqref{eq: traffic}. For any $ i \in [N] $, applying the steady-state probability \eqref{eq:product-form} under the constructed policy $\tilde{\bm{\lambda} }$ gives the availability expression:
\begin{align*}
    \alpha_i (\tilde{\bm{\lambda} } ) = \sum_{ \bm x \in \mathcal{X}_{C,N}: x_i \geq 1} \mathbb{P}_{\bm x} (\tilde{\bm{\lambda} }) =&  
    \frac{\sum_{ \bm x \in \mathcal{X}_{C,N}: x_i \geq 1} \prod_{i=1}^N \left( \alpha_i (\bm \lambda^*)  \right)^{x_i} \prod_{(i,j)\in\A} \frac{1}{x_{ij}! }  \left( v_{ij}(\bm{\lambda}^*)/ \mu^{ij} \right)^{x_{ij}}  }{\sum_{\bm x \in \mathcal X_{C,N}}\prod_{i=1}^N \left( \alpha_i (\bm \lambda^*)  \right)^{x_i} \prod_{(i,j)\in\A} \frac{1}{x_{ij}! }  \left(v_{ij}(\bm{\lambda}^*)/\mu^{ij} \right)^{x_{ij}} } \\ 
    =& \alpha_i(\bm \lambda^* ) \frac{\sum_{ \bm x \in \mathcal{X}_{C-1,N}} \prod_{i=1}^N \left( \alpha_i (\bm \lambda^*)  \right)^{x_i} \prod_{(i,j)\in\A} \frac{1}{x_{ij}! }  \left( v_{ij}(\bm{\lambda}^*) / \mu^{ij} \right)^{x_{ij}}  }{\sum_{\bm x \in \mathcal X_{C,N}}\prod_{i=1}^N \left( \alpha_i (\bm \lambda^*)  \right)^{x_i} \prod_{(i,j)\in\A} \frac{1}{x_{ij}! }  \left( v_{ij}(\bm{\lambda}^*) / \mu^{ij} \right)^{x_{ij}} },
\end{align*}
where the first equality follows from the definition of availability, the second equality follows from \eqref{eq:product-form} with $ w_i = \tilde{\lambda}_i \alpha_i (\bm \lambda^*)$ and $ v_{ij} (\bm \lambda^*) = \tilde{\lambda}^{ij} \alpha_i (\bm{\lambda}^* ) $, and the third equality follows from extracting $\alpha_i (\bm \lambda^* )$ from the numerator and the definition of $\mathcal{X}_{C-1,N}$. Combining the above equality, \eqref{eq:lower-rev-ratio }, and $\R^{sta*}\geq \R^{\tilde{\bm\lambda}}$ implies the desired results of \eqref{eq: rev-ratio-bound}.
\hfill \Halmos  


\end{proof}

Lemma~\ref{lem: RevRatio_lower} generalizes similar lemmas in \cite{besbes2022static} and \cite{elmachtoub2023power} that study the special case with $N = 1$. We remark that the common availability ratio property is a new observation and is inherent to the network setting considered in this paper. Leveraging Lemma~\ref{lem: RevRatio_lower}, we next analyze this common availability ratio and characterize the exact worst-case performance of static pricing. 

\begin{theorem}\label{thm:main-guarantee}
 Let $ \gamma_{C,N}: =  \frac{\sum_{k=0}^{C-1} \frac{1}{k!} (C-\min\{N,C\})^k \binom{\min\{N,C\}+C-k-2}{C-k-1}}{\sum_{k=0}^{C}\frac{1}{k!} (C-\min\{N,C\})^k\binom{\min\{N,C\}+C-k-1}{C-k}}  $. For static policy $ \tilde{\bm{\lambda} } $,
\begin{align*}
 \inf_{\Omega_{C,N}}   \frac{\R^{sta*} }{\R^{*}}  =  \inf_{\Omega_{C,N}}  \frac{\R^{\tilde{\bm\lambda}}}{\R^{*}} =    \gamma_{C,N}  \geq 
 \frac{C}{C + \min\{ C-1, \sqrt{CN} \} } \geq \frac{1}{2}.
\end{align*}
\end{theorem}

Theorem~\ref{thm:main-guarantee} establishes the first universal, non-asymptotic, 50\% guarantee for static pricing against the optimal dynamic pricing, while the best previously known guarantee does not admit any positive universal guarantee. For any $C, N \in \mathbb{Z}_+$,  static pricing achieves a tight $ \gamma_{C,N}$ guarantee, which holds for all concave revenue rate functions $ r^{ij}(\cdot) $ and all travel time distributions with finite means. In particular, for a single location $N=1$, the guarantee reduces to $\gamma_{C,1} = \frac{\sum_{i=0}^{C-1}(C-1)^i/i!}{\sum_{i=0}^{C}(C-1)^i/i! } \geq\frac{15}{19}  $, recovering the one in \cite{elmachtoub2023power}. The explicit lower bound can be equivalently written as $ \frac{C}{C+\min\{C-1,\sqrt{CN}\}} = \max\{ \frac{C}{2C-1}, \frac{\sqrt C}{\sqrt{C} +   \sqrt{\min\{N,C\} } } \} $,
which implies that for fixed $N$, $\gamma_{C,N}$ converges to 1 as $C$ grows to infinity. Finally, we note that when $N\geq C$, $\gamma_{C,N}= \binom{2C-2}{C-1} /\binom{2C-1}{C} = C/(2C-1)$.

\textbf{Comparison to Prior Guarantees.} As mentioned before, our bound is at least $50\%$, achieving the first constant approximation ratio. The previous best guarantee developed in \cite{benjaafar2023pricing} is $ B_{C,N}: = 1 + \frac{N-1}{2C}+\frac{1}{2(N+C)}-\sqrt{\left(\frac{N-1}{2C}+\frac{1}{2(N+C)}\right)^2+\frac{N-1}{C}+\frac{1}{N+C}}$. Our guarantee $ \gamma_{C,N} $ dominates it for any $N$ and $C$.
\begin{corollary}\label{prop: dominate}
    $ \gamma_{C,N} > B_{C,N} $.  
\end{corollary}
Moreover, $B_{C,N} $ is small especially when $N >C $, while our guarantee is always higher than 0.5. For $C \geq 100 $, \cite{banerjee2022pricing} provide another guarantee of $ (1-2\sqrt{\frac{\log C}{C}})(\frac{\sqrt{C\log C}}{\sqrt{C\log C}+N-1}-\frac{3}{\sqrt{C\log C}}) $ under additional offered load assumptions. In Figure~\ref{fig:bound-comparison}, we compare our $ \gamma_{C,N} $ guarantee and those known guarantees for $N = 20, 50$, showing that our bounds are much higher. 

\begin{figure}[h]
\caption{Comparison of the guarantee in Theorem 1 with prior guarantees }\label{fig:bound-comparison}
\centering
\includegraphics[width=0.95\textwidth]{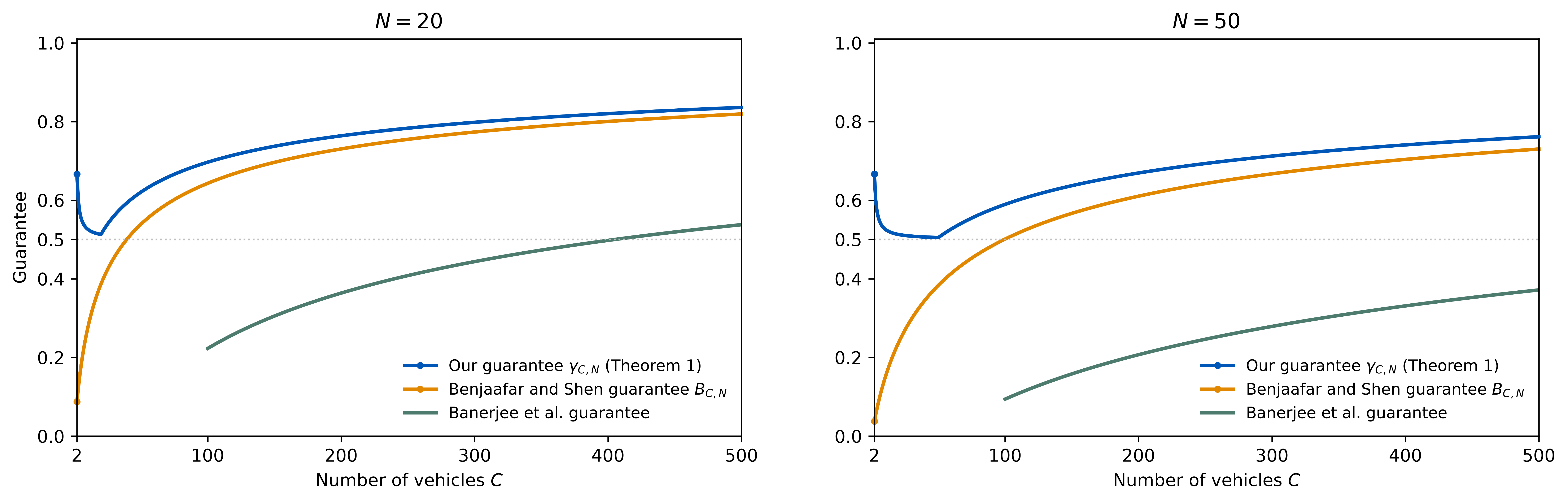}
\end{figure}

\textbf{Intuition on the Improved Guarantee.} 
The improvement of our guarantee $ \gamma_{C,N}$ over the previous guarantees \citep{banerjee2022pricing,benjaafar2023pricing} reflects a different static policy construction. \cite{benjaafar2023pricing} rely on a flow-based static policy with an offered load parameter $\Delta \geq C$ solved from the flow relaxation (\citealp{banerjee2022pricing} use $\Delta = C $):
\begin{equation}\label{opt:fluid-program}
\begin{aligned}
\max_{\bm\lambda: \ 0\leq\lambda^{ij}\leq\Lambda_{ij} } \quad &\sum_{(i,j)\in\A}r^{ij}(\lambda^{ij}) \\
\mathrm{s.t.}\quad &\sum_{j:(i,j)\in\A}\lambda^{ij}=\sum_{k:(k,i)\in\A}\lambda^{ki}&&i\in[N] \\
&\sum_{(i,j)\in\A}\frac{\lambda^{ij}}{\mu^{ij}}\leq \Delta.  
\end{aligned}
\end{equation}
The flow relaxation accepts all trip requests and constrains only the total offered transit load. It also imposes a balanced network constraint: $ \sum_{j:(i,j)\in\A}\lambda^{ij}=\sum_{k:(k,i)\in\A}\lambda^{ki} $, $ i \in [N] $, meaning that each location is equally popular as both an origin and a destination under this static policy. Let $ \bm \lambda_{\Delta} = ( \lambda^{ij}_{\Delta} )_{(i,j) \in \A} $ denote the static policy solved from \eqref{opt:fluid-program} and $ \alpha_i ( \bm \lambda_{\Delta} ) $ be its availability at location $i \in [N] $. A balanced network implies that the availability is the same across all locations, i.e., $ \alpha_i ( \bm \lambda_{\Delta} ) = \alpha_j ( \bm \lambda_{\Delta} ) $ for all $i,j \in [N]$. \cite{benjaafar2023pricing} then show that the revenue ratio between $ \bm \lambda_{\Delta}  $ and the optimal dynamic policy is lower bounded by the common availability $ \alpha_i ( \bm \lambda_{\Delta} ) $.

However, under the optimal dynamic policy, vehicle availability may be highly unbalanced as a few locations are inherently more popular than others based on the reward function $ r^{ij}(\cdot) $ and mean travel time $ 1/\mu^{ij} $. The balanced static policy from \eqref{opt:fluid-program} may perform poorly in these settings. In contrast, our constructed static policy $ \tilde{\lambda}^{ij}= v_{ij} (\bm \lambda^*)/ \alpha_i(\bm\lambda^*)  $ in \eqref{eq:constructed-policy} respects the unbalanced network of the optimal dynamic policy by using its conditional average effective arrival rate when origin $i$ has a vehicle. Moreover, as shown in Lemma \ref{lem: RevRatio_lower}, the availability under our constructed static policy is in relative proportion to the optimal dynamic policy, i.e., $ \frac{\alpha_i(\tilde{\bm\lambda})}{\alpha_i(\bm\lambda^*)} = \frac{\alpha_j(\tilde{\bm\lambda})}{\alpha_j(\bm\lambda^*)} $ for all $i,j\in [N] $. 

\subsection{Proof of Theorem~\ref{thm:main-guarantee} }

For any $\bm\alpha \in \mathbb R_+^N$ and $\rho \geq 0$, we define an auxiliary function
\begin{align}\label{eq:G_C}
    G_C(\bm\alpha, \rho ):= \sum_{k=0 }^{C} \frac{1}{k!} \rho^k \sum_{ x_1 + \cdots + x_N = C - k  }  \prod_{i=1}^N  \alpha_i^{x_i} .
\end{align}
Let $\rho(\bm \lambda^* ):= \sum_{(i,j) \in \A } v_{ij}(\bm \lambda^*)/\mu^{ij}  $ denote the total travel load of the transit queue. We observe that 
\begin{align*}
 \sum_{\bm x \in \mathcal X_{C,N} } \prod_{i=1}^N \left( \alpha_i (\bm \lambda^*)  \right)^{x_i} \prod_{(i,j)\in\A} \frac{1}{x_{ij}! }  \left(\frac{v_{ij}(\bm{\lambda}^*) }{\mu^{ij}} \right)^{x_{ij}}
 &=  \sum_{k=0 }^{C} \frac{1}{k!} \left( \sum_{(i,j) \in \A } \frac{v_{ij}(\bm \lambda^*) }{\mu^{ij}}  \right)^k \sum_{ x_1 + \cdots + x_N = C - k  }  \prod_{i=1}^N \left( \alpha_i (\bm \lambda^*)  \right)^{x_i}  \\
&= G_C( \bm{\alpha}(\bm \lambda^*), \rho (\bm \lambda^*) ),
\end{align*}
where the first equality follows from the multinomial theorem and the definition of $\mathcal{X}_{C,N} $.  From \eqref{eq: rev-ratio-bound} in Lemma~\ref{lem: RevRatio_lower}, our goal is to lower bound
\begin{align*}
  \frac{\sum_{\bm x \in \mathcal X_{C-1, N} }\prod_{i=1}^N \left( \alpha_i (\bm \lambda^*)  \right)^{x_i} \prod_{(i,j)\in\A} \frac{1}{x_{ij}! }  \left( \frac{1}{\mu^{ij}} v_{ij}(\bm{\lambda}^*) \right)^{x_{ij}} }{\sum_{\bm x \in \mathcal X_{C,N} }\prod_{i=1}^N \left( \alpha_i (\bm \lambda^*)  \right)^{x_i} \prod_{(i,j)\in\A} \frac{1}{x_{ij}! }  \left( \frac{1}{\mu^{ij}}  v_{ij}(\bm{\lambda}^*)  \right)^{x_{ij}} }  =    \frac{G_{C-1}( \bm{\alpha}(\bm \lambda^*), \rho (\bm \lambda^*)  ) }{G_C( \bm{\alpha}(\bm \lambda^*), \rho (\bm \lambda^*)  ) } : = R(\bm{\alpha}(\bm \lambda^*), \rho (\bm \lambda^*)  ). 
\end{align*}
Using Little's law on the optimal dynamic policy, we impose a natural constraint described in Lemma \ref{lem:little's law}.

\begin{lemma}\label{lem:little's law}
From Little's law on the optimal dynamic policy $\bm \lambda^*$, $ \sum_{i=1}^N\alpha_i(\bm\lambda^*)+ \rho (\bm \lambda^*)  \leq C. $
\end{lemma}
Therefore, a lower bound on $ R(\bm{\alpha}(\bm \lambda^*), \rho (\bm \lambda^*)  )$ reduces to the following minimization problem:
\begin{align}\label{min: ratio}
 \min_{ \bm \alpha \in [0,1]^N,  \ \rho \geq 0 } \ & R( \bm \alpha, \rho)  \qquad \textrm{s.t.}  \quad \sum_{i=1}^N \alpha_i + \rho \leq C .
\end{align}
We next establish structural properties of $R(\bm \alpha, \rho)$ in Lemma~\ref{lem:availability-ratio-property}, indicating the solutions of \eqref{min: ratio}.

\begin{lemma}\label{lem:availability-ratio-property} 
Let $\bm e_i$ denote the $i$-th unit vector. For all $(\bm\alpha,\rho)\in \mathbb R_+^N \times\mathbb R_+$, the following hold.
\begin{enumerate}
\item[(i)] For any $\delta\geq0$, $ R( \bm \alpha,\rho + \delta) \leq R(\bm \alpha,\rho)  $.
\item[(ii)] For any $i\in[N]$ and $0\leq\delta\leq \rho $, $ R(\bm\alpha+\delta\bm e_i,\rho-\delta) \leq R(\bm \alpha,\rho) $.
\item[(iii)] For any $i,j\in[N]$ with $\alpha_i\geq\alpha_j$ and $0\leq\delta\leq  \alpha_j, $, $ R(\bm\alpha+\delta\bm e_i-\delta\bm e_j,\rho) \leq R(\bm \alpha,\rho)$.
\end{enumerate}
\end{lemma}

Lemma~\ref{lem:availability-ratio-property} implies that the optimal solution of \eqref{min: ratio} has a simple structure. Part (i) makes the constraint binding, part (iii) pushes the $\alpha_i$'s toward $0$ or $1$, and part (ii) transfers mass from $\rho$ to any $\alpha_i<1$, which yields the following lemma.

\begin{lemma}\label{lem: solution}
There exists an optimal solution to \eqref{min: ratio} with $\sum_{i=1}^N \hat{\alpha}_i =  \min \{N, C\} $, $ \hat{\alpha}_i\in \{0,1 \} $ for all $i \in [N]$, and $ \hat{\rho} = C -  \min \{N, C\}$.
\end{lemma}

Because $(\bm\alpha (\bm\lambda^*), \rho(\bm\lambda^*))$ is a feasible solution of \eqref{min: ratio}, we obtain that
\begin{align}
R(\bm{\alpha}(\bm \lambda^*), \rho (\bm \lambda^*)  ) \geq R( \hat{\bm\alpha}, \hat{\rho} ) 
=& \frac{ \sum_{k=0 }^{C-1} \frac{1}{k!} \hat{\rho}^k \sum_{ x_1 + \cdots + x_N = C -1 - k  }  \prod_{i=1}^N  \hat{\alpha}_i^{x_i}  }{ \sum_{k=0 }^{C} \frac{1}{k!} \hat{\rho}^k \sum_{ x_1 + \cdots + x_N = C - k  }  \prod_{i=1}^N  \hat{\alpha}_i^{x_i} } \notag  \\
=&  \frac{\sum_{k=0}^{C-1} \frac{1}{k!} (C-\min\{N,C\})^k \binom{\min\{N,C\}+C-k-2}{C-k-1}}{\sum_{k=0}^{C}\frac{1}{k!} (C-\min\{N,C\})^k\binom{\min\{N,C\}+C-k-1}{C-k}} = \gamma_{C,N} \label{eq: gamma} . 
\end{align}
The inequality holds by Lemmas \ref{lem:little's law} and \ref{lem: solution}. The first equality follows from \eqref{eq:G_C}. The second equality follows from $\hat{\rho} = C -  \min \{N, C\} $, the fact that exactly $\min\{N,C\}$ components of $\hat{\bm\alpha}$ equal 1 and weak composition formula for the number of $\min\{N,C\}$ nonnegative integers summing to $C-k$ or $C-1-k$. This proves $ \frac{\R^{\tilde{\bm\lambda}}}{\R^{*}} \geq \gamma_{C,N} $. We then show $ \gamma_{C,N} $ is tight by presenting an example with a matching lower bound.

\begin{lemma}\label{lem:tightness}
Consider a directed cycle $\A=\{(1,2),\ldots,(N-1,N),(N,1)\}$. Let $H_{ij}$ be exponential with mean $1/\epsilon$, let $p^{ij}(\lambda)= \frac{b N}{m\epsilon} + \frac{a}{\lambda}$ for the first $\min\{N,C\}$ arcs along the cycle for some constants $a,b$, and let $p^{ij}(\lambda)=\epsilon (1-\lambda) $ on the remaining arcs. Then, $\lim_{\epsilon\to0}\frac{\R^{\sta*}}{\R^*}=\gamma_{C,N}$.
\end{lemma}

Finally, the lemma below provides the lower bound on $\gamma_{C,N}$ and completes the proof of Theorem~\ref{thm:main-guarantee}.
\begin{lemma}\label{lem: lowerbound-gamma}
    $  \gamma_{C,N} \geq \frac{C}{C+\min\{C-1,\sqrt{CN}\}} = \max\{ \frac{C}{2C-1}, \frac{\sqrt C}{\sqrt{C} +   \sqrt{\min\{N,C\} } } \} $. 
\end{lemma}

\subsection{Zero Travel Times}
We also consider an extreme case in which all trip travel times are zero. Previous works
\citep{waserhole2016pricing,banerjee2022pricing,benjaafar2023pricing} apply different methods to show a guarantee of $\frac{C}{C+N-1}$ for static pricing. Let $\Omega^0_{C,N}$ denote the set of instances in which all travel times are zero. We establish that the tight guarantee of static pricing is $ \frac{C}{C+\min\{C,N\}-1}$. 

\begin{theorem}\label{thm:zero-travel}
$ \inf_{\Omega^0_{C,N}}\frac{\R^{sta*}}{\R^*} = \inf_{\Omega^0_{C,N}} \frac{\R^{\tilde{\bm\lambda}}}{\R^*} = \frac{C}{C+\min\{C,N\}-1}. $
\end{theorem}

When $N<C$, our guarantee is exactly $\frac{C}{C+N-1}$, while it improves to $\frac{C}{2C-1}$ when $N\geq C$. The proof follows the same argument as that of Theorem~\ref{thm:main-guarantee}, using an additional fact that the total travel load satisfies $\rho(\bm\lambda^*)=0$ in \eqref{eq: gamma}.


\section{Optimal Static Policy}\label{sec:opt-static}
Our previous analysis is based on a constructed static policy from the optimal dynamic policy. The implementation of such constructed static policy is hard as the optimal dynamic policy is generally unknown. In practice, one can directly compute the optimal static policy, which we show can be efficiently computed.

Fix a static policy $\bm\lambda=(\lambda^{ij})_{(i,j)\in\A}$ that induces a strongly connected network and recall that $ \lambda_i=\sum_{j:(i,j)\in\A}\lambda^{ij} $, $i \in [N] $. For notation simplicity, let $\omega_i$ and $\lambda^{ij}$ be a solution of the traffic equations \eqref{eq: traffic}, i.e., $ \omega_i\lambda_i=\sum_{k:(k,i)\in\A}\omega_k\lambda^{ki} $. Then, from \eqref{eq:product-form}, the availability at location $i $, $\alpha_i( \bm \lambda )$, is
\begin{align*}
\alpha_i(\bm\lambda) =
\sum_{\bm x\in\mathcal X_{C,N}:x_i \geq 1} \mathbb{P}_{\bm x}(\bm\lambda) 
= \frac{\sum_{ \bm x \in \mathcal{X}_{C,N}: x_i \geq 1} \prod_{i=1}^N \omega_i^{x_i} \prod_{(i,j)\in\A} \frac{1}{x_{ij}! }  \left( \omega_i \lambda^{ij}/ \mu^{ij} \right)^{x_{ij}}  }{\sum_{\bm x \in \mathcal X_{C,N}}\prod_{i=1}^N \omega_i^{x_i} \prod_{(i,j)\in\A} \frac{1}{x_{ij}! }  \left(\omega_i  \lambda^{ij}/\mu^{ij} \right)^{x_{ij}} }. 
\end{align*}
Thus, the long-run average static revenue under $\bm\lambda=(\lambda^{ij})_{(i,j)\in\A}$ is expressed as
\begin{align}\label{eq:static-rev}
    \R^{sta}(\bm\lambda) \! = \! \sum_{(i,j)\in\A} \! \alpha_i(\bm\lambda) r^{ij}(\lambda^{ij}) \! = \! \frac{\sum_{ \bm x \in \mathcal{X}_{C-1, N}} \! \prod_{i=1}^N \omega_i^{x_i} \prod_{(i,j)\in\A} \frac{1}{x_{ij}! }  \left( \omega_i \lambda^{ij}/ \mu^{ij} \right)^{x_{ij}}  }{\sum_{\bm x \in \mathcal X_{C,N}} \! \prod_{i=1}^N \omega_i^{x_i}   \prod_{(i,j)\in\A} \frac{1}{x_{ij}! }  \left(\omega_i  \lambda^{ij}/\mu^{ij} \right)^{x_{ij}} } \sum_{(i,j)\in\A} \omega_i r^{ij}(\lambda^{ij}) ,
\end{align}
where the second equality follows from extracting one  $\omega_i$ in the numerator of $\alpha_i(\bm\lambda)$. Note that $ \omega_i $ and $\lambda^{ij}$ need to satisfy the bilinear constraints $  \sum_{ j: (i,j) \in \A } \omega_i \lambda^{ij} = \omega_i\lambda_i = \sum_{k:(k,i)\in\A} \omega_k\lambda^{ki} $, $ i \in [N] $. We next establish the computational tractability of the optimal static pricing policy.

\begin{theorem}\label{thm:convex-static-program}
The optimal static policy maximizing \eqref{eq:static-rev} can be computed by solving a convex optimization problem.
\end{theorem}

\begin{proof}{Proof.}
For $ \bm z = (z_{ij})_{(i,j) \in A} $,  we  define  $   Q_C(\bm \omega, \bm z): = \frac{\sum_{\bm x \in \mathcal X_{C,N}} \prod_{i=1}^N \omega_i^{x_i}  \prod_{(i,j)\in\A} \frac{1}{x_{ij}! }  \left( z_{ij}/\mu^{ij} \right)^{x_{ij}}}{\sum_{ \bm x \in \mathcal{X}_{C-1,N}} \prod_{i=1}^N \omega_i^{x_i} \prod_{(i,j)\in\A} \frac{1}{x_{ij}! }  \left( z_{ij}/ \mu^{ij} \right)^{x_{ij}} } . $
We first show in Lemma~\ref{lem:static-equivalence} that the optimal static revenue equals the optimal objective value of the following optimization problem:
\begin{equation}\label{opt:convex-program}
\begin{aligned}
\max_{\bm\omega \geq \bm{0} ,\bm z}\quad
&\sum_{(i,j)\in\A}\omega_i r^{ij}\left(\frac{z_{ij}}{\omega_i}\right)\\
\mathrm{s.t.}\quad
&\sum_{j:(i,j)\in\A}z_{ij}= \sum_{k:(k,i) \in\A} z_{ki} &&i\in[N],\\
&0\leq z_{ij}\leq\Lambda_{ij}\omega_i &&(i,j)\in \A,\\
&Q_C(\bm\omega,\bm z)\leq1.
\end{aligned} 
\end{equation}

\begin{lemma}\label{lem:static-equivalence}
The optimal static revenue $\R^{sta*}$ equals the optimal objective value of \eqref{opt:convex-program}.
\end{lemma}

Since $r^{ij}(\cdot)$ is concave, the objective of \eqref{opt:convex-program} is concave by the perspective property. Moreover, we observe in Lemma~\ref{lem:Q_C convex} that $Q_C(\bm\omega,\bm z)$ is convex. Therefore, \eqref{opt:convex-program} is a convex optimization problem, completing the proof of Theorem \ref{thm:convex-static-program}.

\begin{lemma}\label{lem:Q_C convex}
   $Q_C (\bm \omega, \bm z) $ is convex.
\end{lemma}

\begin{proof}{Proof of Lemma~\ref{lem:static-equivalence}.}
($\Rightarrow$)
For any static policy $\bm\lambda$ and $ \bm \omega $ defined in \eqref{eq:static-rev}, let $z_{ij}=\omega_i\lambda^{ij}$.  By the traffic equations and $0 \leq \lambda^{ij} \leq \Lambda_{ij}$, we have that $ \sum_{j:(i,j)\in\A}z_{ij}
=\omega_i\lambda_i
=\sum_{k:(k,i)\in\A}\omega_k\lambda^{ki}
=\sum_{k:(k,i)\in\A}z_{ki} $ and  $
0\leq z_{ij}\leq\Lambda_{ij}\omega_i $. Thus, $(\bm\omega,\bm z)$ satisfies the first two constraints of \eqref{opt:convex-program}.  To satisfy the last constraint, let $t=1/Q_C(\bm\omega,\bm z)>0$. Since $Q_C$ is homogeneous of degree one, $Q_C(t\bm\omega,t\bm z)=1$. Scaling both $\bm\omega$ and $\bm z$ by $t$ also preserves the first two constraints, implying that $(t\bm\omega, t \bm z)$ is feasible for \eqref{opt:convex-program}, with objective value $ \sum_{(i,j)\in\A}t\omega_i r^{ij}\left(\frac{tz_{ij}}{t\omega_i}\right)
=\frac{\sum_{(i,j)\in\A}\omega_i r^{ij}(\lambda^{ij})}{Q_C(\bm\omega,\bm z)}
=\R^{sta}(\bm\lambda) $, where the last equality follows from \eqref{eq:static-rev}. Thus, the optimal objective value of \eqref{opt:convex-program} is at least $\R^{sta*}$.

($\Leftarrow$) 
Conversely, consider any feasible $(\bm\omega, \bm z)$ and let $\lambda^{ij}=z_{ij}/\omega_i$ when $\omega_i>0$ and $\lambda^{ij}=0$ otherwise. The first two constraints ensure that $\bm\lambda$ satisfies $0\leq\lambda^{ij}\leq\Lambda_{ij}$ and $\omega_i\lambda_i=\sum_{k:(k,i)\in\A}\omega_k\lambda^{ki}$. By \eqref{eq:static-rev}, the corresponding static revenue is $ \R^{sta}(\bm\lambda)
=\frac{\sum_{(i,j)\in\A}\omega_i r^{ij}(z_{ij}/\omega_i)}{Q_C(\bm\omega,\bm z)}
\geq\sum_{(i,j)\in\A}\omega_i r^{ij}(z_{ij}/\omega_i) $,  where the inequality follows from $Q_C(\bm\omega,\bm z) \leq 1$. Since $ \R^{sta}(\bm\lambda) \leq \R^{sta*} $,  the optimal objective value of \eqref{opt:convex-program} is at most $\R^{sta*}$. 
\end{proof}


\end{proof}

\begin{proof}{Proof of Lemma~\ref{lem:Q_C convex}.}

We show that $Q_C(\bm\omega,\bm z)$ is convex (with the continuous extension $Q_C(\bm0,\bm0)=0$). For any integer $ m \ge1 $, let $h_k(\omega_1,\ldots, \omega_m) := \sum_{x_1+\cdots+x_m=k}\prod_{l=1}^m \omega_l^{x_l} $. Theorem 3.26 in \cite{chen2001fundamentals} shows that $\frac{h_C(\omega_1, \ldots,\omega_m)}{h_{C-1}(\omega_1, \ldots,\omega_m)} $ is jointly convex in $(\omega_1,\ldots, \omega_m)$.  For integer $M\geq1$, we consider 
\begin{align*}
Q_C^{M}(\bm\omega,\bm z) :=
\frac{ h_C \left(\omega_1,\ldots,\omega_N,
\frac{1}{M}\sum_{(i,j)\in\A}\frac{z_{ij}}{\mu^{ij}},\ldots,
\frac{1}{M}\sum_{(i,j)\in\A}\frac{z_{ij}}{\mu^{ij}} \right) }{
h_{C-1}\left(\omega_1,\ldots,\omega_N,
\frac{1}{M} \sum_{(i,j)\in\A}\frac{z_{ij}}{\mu^{ij}},\ldots,
\frac{1}{M}\sum_{(i,j)\in\A}\frac{z_{ij}}{\mu^{ij}}\right) },
\end{align*}
where $ \frac{1}{M}\sum_{(i,j)\in\A}\frac{z_{ij}}{\mu^{ij}} $ is repeated $M$ times. Since $Q_C^M$ is obtained by composing the convex function $h_C/h_{C-1}$ with a linear mapping of $(\bm \omega, \bm z)$, $Q_C^M$ is convex in $(\bm\omega,\bm z)$. We next show that $Q_C^{M}$ converges to $Q_C$. For $k\geq0$, note that
\begin{align*}
    h_k \left(\omega_1,\ldots,\omega_N,
\frac{1}{M}\sum_{(i,j)\in\A}\frac{z_{ij}}{\mu^{ij}},\ldots,
\frac{1}{M}\sum_{(i,j)\in\A}\frac{z_{ij}}{\mu^{ij}} \right) = \sum_{i=0}^k h_{k-i}(\omega_1,\ldots,\omega_N) \binom{M+i-1}{i} \left[ \sum_{(i,j)\in\A} \frac{z_{ij}}{M\mu^{ij}} \right]^i.
\end{align*}
Because $ \lim_{M \to \infty} \binom{M+i-1}{i}M^{-i} = 1/i!$, we have that
\begin{align*}
\lim_{M \to \infty}  h_k \left(\omega_1,\ldots,\omega_N,
\frac{1}{M}\sum_{(i,j)\in\A}\frac{z_{ij}}{\mu^{ij}},\ldots,
\frac{1}{M}\sum_{(i,j)\in\A}\frac{z_{ij}}{\mu^{ij}} \right) =& \sum_{i=0}^k h_{k-i}(\omega_1,\ldots,\omega_N) \frac{1}{i!} \left( \sum_{(i,j)\in\A} \frac{z_{ij}}{\mu^{ij}} \right)^i \\
=& \sum_{\bm x\in\mathcal X_{k,N}} \prod_{i=1}^N \omega_i^{x_i} \prod_{(i,j)\in\A} \frac{1}{x_{ij}!} \left(\frac{z_{ij}}{\mu^{ij}}\right)^{x_{ij}},
\end{align*}
where the last equality follows from the multinomial theorem. Applying the above convergence with $k=C$ and $k=C-1$ yields that $ \lim_{M\to\infty}Q_C^{M}(\bm\omega,\bm z) = Q_C(\bm\omega,\bm z) $. Thus, for any two points $(\bm\omega^1,\bm z^1)$ and $(\bm\omega^2,\bm z^2)$ and any $\theta\in[0,1]$, convexity of $Q_C^{M}$ shows that
\begin{align*}
Q_C\left( \theta\bm\omega^1+(1-\theta)\bm\omega^2,
\theta\bm z^1+(1-\theta)\bm z^2 \right)
=& \lim_{M\to\infty} Q_C^{M}\left( \theta\bm\omega^1 + (1-\theta)\bm\omega^2, \theta\bm z^1+(1-\theta)\bm z^2 \right)\\
\leq & \theta \cdot \lim_{M\to\infty}Q_C^{M}(\bm \omega^1, \bm z^1) + (1-\theta) \cdot \lim_{M\to\infty}Q_C^{M}(\bm\omega^2,\bm z^2)\\
=& \theta Q_C(\bm\omega^1,\bm z^1) + (1-\theta)Q_C(\bm \omega^2, \bm z^2).
\end{align*}
Hence, $Q_C(\bm \omega, \bm z)$ is convex. \hfill \Halmos

\end{proof}

For any optimal solution $(\bm\omega^*,\bm z^*)$ of \eqref{opt:convex-program}, an optimal static pricing policy is recovered as $ \lambda^{ij} = z_{ij}^* / \omega_i^*$ when $ \omega_i^* >0 $, and $\lambda^{ij}=0$ otherwise.
Although $Q_C(\bm\omega,\bm z)$ is defined through the product-form expressions, it can be evaluated without enumerating the state space $\mathcal{X}_{C,N}$. Initialize $G_k^{(0)}=(\sum_{(i,j)\in\A} \frac{ z_{ij} }{ \mu^{ij} })^k/k! $ and recursively compute $ G_k^{(i)} = G_k^{(i-1)} + \omega_iG_{k-1}^{(i)} $, $i=1,\ldots,N, \ k=0,\ldots,C$ with $G_{-1}^{(i)}=0$. Then, $Q_C(\bm\omega,\bm z) = G_C^{(N)}/G_{C-1}^{(N)}$. Hence, $Q_C$ and its gradient can be evaluated in $O(NC+|\A|)$, which indicates that \eqref{opt:convex-program} can be solved using standard convex optimization methods.

\subsection{Line Search over the Flow Relaxation}\label{sec:balanced-comparison}

\citet{benjaafar2023pricing} propose a line-search method to compute a balanced static policy. They solve the flow relaxation \eqref{opt:fluid-program} multiple times by varying $\Delta\in[0,\sum_{(i,j)\in\A}\Lambda^{ij}/\mu^{ij}]$ and select the resulting static policy with the highest revenue. They show that the line search achieves the highest revenue among static policies satisfying the balance constraints $\sum_{j:(i,j)\in\A}\lambda^{ij}=\sum_{k:(k,i)\in\A}\lambda^{ki}$, $i\in[N]$. We refer to this revenue as the best balanced static revenue, which may be lower than the optimal static revenue. In contrast, our formulation \eqref{opt:convex-program} computes the optimal static policy without imposing the balance constraints by solving a single convex optimization problem.

We conduct numerical experiments to quantify the improvement of the optimal static policy over the best balanced static policy. We consider linear demand, $p^{ij}(\lambda)=b^{ij}-a^{ij}\lambda$, and exponential demand, $p^{ij}(\lambda)=a^{ij}\log(b^{ij}/(a^{ij}\lambda))$ with $ \Lambda_{ij} =b^{ij}/a^{ij} $. We use $N\in\{5,10,15,20\}$ and $C\in\{5,20,40\}$ and generate 100 independent instances for each pair $(N,C)$ and each demand family. Independently across arcs, we draw $a^{ij}$ uniformly from $[0.1,5]$, $b^{ij}$ uniformly from $[0.5,10]$, and $\mu^{ij}$ uniformly from $[0.2,5]$. For each instance, we compute the optimal static revenue from \eqref{opt:convex-program} and the best balanced static revenue using line search of $\Delta$ over \eqref{opt:fluid-program}.


\begin{table}[ht]
\centering
\caption{Revenue ratio of the best balanced static policy relative to the optimal static policy}
\label{tab:offered-balanced}
\begin{threeparttable}
\small
\setlength{\tabcolsep}{2.5pt}
\begin{tabular}{cccccc@{\hspace{0.8em}\vrule\hspace{0.8em}}cccccc}
\toprule
& & \multicolumn{2}{c}{Linear} & \multicolumn{2}{c}{Exponential}
& & \multicolumn{2}{c}{Linear} & \multicolumn{2}{c}{Exponential} \\
\cmidrule(lr){3-4}\cmidrule(lr){5-6}
\cmidrule(lr){9-10}\cmidrule(lr){11-12}
$N$ & $C$ & Min. & Avg. & Min. & Avg.
& $N$ & $C$ & Min. & Avg. & Min. & Avg. \\
\midrule
5  & 5  & 80.66\% & 94.73\% & 89.30\% & 97.04\%
& 15 & 5  & 80.15\% & 94.75\% & 95.62\% & 97.99\% \\
   & 20 & 92.02\% & 97.90\% & 96.97\% & 99.08\%
&    & 20 & 88.10\% & 96.59\% & 98.04\% & 99.12\% \\
   & 40 & 95.77\% & 98.68\% & 98.90\% & 99.55\%
&    & 40 & 95.59\% & 98.08\% & 99.19\% & 99.56\% \\
\midrule
10 & 5  & 71.08\% & 94.29\% & 94.70\% & 97.76\%
& 20 & 5  & 82.55\% & 94.80\% & 96.49\% & 98.25\% \\
   & 20 & 92.82\% & 97.43\% & 97.40\% & 99.13\%
&    & 20 & 89.35\% & 96.76\% & 98.75\% & 99.26\% \\
   & 40 & 94.94\% & 98.30\% & 99.02\% & 99.54\%
&    & 40 & 95.19\% & 97.93\% & 98.99\% & 99.54\% \\
\bottomrule
\end{tabular}
\begin{tablenotes}[flushleft]
\footnotesize
\item \textit{Note.} We consider a complete directed network with $N(N-1)$ arcs for each $N$.
\end{tablenotes}
\end{threeparttable}
\end{table}

Table~\ref{tab:offered-balanced} reports the best balanced static revenue as a percentage of the optimal static revenue. The average ratio ranges from $94.29\%$ to $99.56\%$, while the smallest observed ratio is $71.08\%$. For every fixed $N$ and both demand families, the minimum and average ratios increase with $C$. The ratios are also consistently higher under exponential demand than under linear demand.



\section{Supply Redirection}

In this section, we extend our guarantee of static policy in Theorem~\ref{thm:main-guarantee} to supply redirection. Following \cite{banerjee2022pricing}, a vehicle may either remain at the destination or be redirected (sent empty) to another location after completing a customer trip. Supply redirection provides an additional means of mitigating spatial imbalance in vehicle availability. We assume that each redirection from location $i$ to location $j$ incurs a cost $c_{ij}\geq 0$ and has a travel time following a general distribution with mean $1/\nu^{ij}$.


With supply redirection, the state space further tracks the number of vehicles redirecting on each pair $(i,j)$ and their travel time spent so far. We denote the resulting system state space by $\tilde{\mathcal S}_{C,N}$. A dynamic policy chooses both the pricing policy and the redirection decisions based on the system state. Let $\eta_{\bm s}^{ij}$ denote the probability that a vehicle completing a customer trip at $i$ is immediately redirected to $j$ in state $s\in\tilde{\mathcal S}_{C,N}$, where $\sum_{j \in [N] }\eta_{\bm s}^{ij} \leq 1$. In contrast, a static policy offers fixed trip prices and fixed redirection probabilities  at all times, independent of the system state. For any pricing and redirection policy $(\bm\lambda, \bm\eta)$, let $u_{ij}(\bm\lambda, \bm\eta)$ denote the long-run average throughput of redirection trips from $i$ to $j$. The objective is revenue minus redirection cost:
\begin{align*}
\R(\bm\lambda, \bm\eta) = \sum_{(i,j)\in\A} \int_{\tilde{\mathcal S}_{C,N}} \mathbf{1}\{x_i \geq 1\}r^{ij}(\lambda_{\bm s}^{ij}) d \Pi^{\bm\lambda, \bm\eta}( \bm{s})-\sum_{i,j\in [N] }c_{ij} u_{ij}(\bm\lambda,\bm\eta). 
\end{align*}
We assume that an optimal dynamic policy induces an ergodic system and a strongly connected network and let $(\bm\lambda^*,\bm\eta^*)$ denote such an optimal policy. Similar to Section~\ref{sec:main-result}, we show the guarantee of a static policy through a constructed one from the optimal dynamic policy. Formally, the static price $ \tilde\lambda^{ij} $ (average effective arrival rates when vehicles are available at $i$) and redirection probability $\tilde\eta^{ij} $ (average redirection probability upon trip completion at $i$) are defined as
\begin{align}\label{eq:rebalancing-static-construction} 
\tilde\lambda^{ij}=\frac{v_{ij}(\bm\lambda^*,\bm\eta^*)}{\alpha_i(\bm\lambda^*,\bm\eta^*)}, \quad (i,j) \in \A, \quad 
\tilde\eta^{ij}=\frac{u_{ij}(\bm\lambda^*,\bm\eta^*)}{\sum_{k:(k,i)\in\A}v_{ki}(\bm\lambda^*,\bm\eta^*)}, \quad i,j\in[N],
\end{align}
where $v_{ij}(\bm\lambda^*,\bm\eta^*)= \int_{\tilde{\mathcal{S} }_{C,N}} \mathbf 1\{x_i\geq 1\}\lambda_{\bm s}^{ij*}  d \Pi^{\bm\lambda^*,  \bm\eta^*} \! (\bm s ) $ is the throughput of $(i,j)$ trip and $ \alpha_i(\bm\lambda^*,\bm\eta^*) $ is the availability at location $i$ under the optimal dynamic policy. We set $\tilde\eta^{ij}=0$ if  the denominator is zero. Let $ \tilde{\bm \lambda} = (\tilde{\lambda}^{ij})_{(i,j) \in \A} $ and $\tilde{\bm \eta} = (\tilde\eta^{ij})_{i,j \in [N]} $ denote this static policy. Despite the additional dynamic redirection decisions, we show that a static policy still admits the same $ \gamma_{C,N} $ guarantee.


\begin{theorem}\label{thm:supply-redirection}
For the joint pricing and redirection problem, our static policy $(\tilde{\bm\lambda}, \tilde{\bm\eta})$ guarantees
\begin{align*}
\frac{\R(\tilde{\bm\lambda}, \tilde{\bm\eta})}{\R( \bm \lambda^*, \bm \eta^* )} \geq\gamma_{C,N} \geq\frac{C}{2C-1}>\frac{1}{2}. 
\end{align*}
\end{theorem}


\begin{proof}{Proof.}
According to the definitions of $ \R(\tilde{\bm\lambda}, \tilde{\bm\eta}) $ and $\R( \bm \lambda^*, \bm \eta^* )$,
\begin{align*}
 \frac{\R(\tilde{\bm\lambda}, \tilde{\bm\eta})}{\R( \bm \lambda^*, \bm \eta^* )} 
 =& \frac{ \sum_{(i,j)\in \A} \alpha_i(\tilde{\bm\lambda}, \tilde{\bm\eta} ) r^{ij}( \tilde{\lambda}^{ij})  - \sum_{i,j\in [N] } c_{ij} u_{ij}(\tilde{\bm\lambda}, \tilde{\bm\eta}) }{ \sum_{(i,j)\in\A} \int_{\tilde{\mathcal S}_{C,N}} \mathbf{1}\{x_i \geq 1\}r^{ij}(\lambda_{\bm s}^{ij*}) d \Pi^{\bm{\lambda}^*, \bm{\eta}^* }\! (\bm{s}) - \sum_{i,j\in [N] }c_{ij} u_{ij}(\bm{\lambda}^*,\bm{\eta}^*)}  \\
\geq& \frac{ \sum_{(i,j)\in \A} \alpha_i(\tilde{\bm\lambda}, \tilde{\bm\eta} ) r^{ij}( \tilde{\lambda}^{ij})  - \sum_{i,j\in [N] } c_{ij} u_{ij}(\tilde{\bm\lambda}, \tilde{\bm\eta})   }{\sum_{(i,j)\in\A}\alpha_i(\bm\lambda^*, \bm\eta^*)r^{ij}(\tilde\lambda^{ij}) - \sum_{i,j \in [N] } c_{ij} u_{ij}(\bm\lambda^*,\bm\eta^*),  }  \\
=& \frac{G_{C-1}(\bm\alpha(\bm\lambda^*,\bm\eta^*),\rho(\bm\lambda^*,\bm\eta^*))}{G_C(\bm\alpha(\bm\lambda^*,\bm\eta^*),\rho(\bm\lambda^*,\bm\eta^*))} 
\end{align*}
where the inequality follows from Jensen's inequality and \eqref{eq:rebalancing-static-construction}, and the second equality follows from Lemma~\ref{lem:redirection-common-ratio} below, which shows that the availability and throughput ratios are the same across all locations and arcs, analogous to the observation in Lemma~\ref{lem: RevRatio_lower}. Let $\rho(\bm \lambda^*, \bm\eta^*  ):= \sum_{(i,j) \in \A } \frac{v_{ij}(\bm \lambda^*, \bm\eta^* )}{\mu^{ij} } + \sum_{i,j \in [N]} \frac{u_{ij}(\bm\lambda^*,\bm\eta^*)}{\nu^{ij} }  $ and recall that $G_C(\cdot)$ is defined in \eqref{eq:G_C}.


\begin{lemma}\label{lem:redirection-common-ratio}
Under the constructed static policy $(\tilde{\bm\lambda},\tilde{\bm\eta})$, it holds that
\begin{align*}
\frac{\alpha_i(\tilde{\bm\lambda}, \tilde{\bm\eta})}{\alpha_i(\bm\lambda^*,\bm\eta^*)}=\frac{G_{C-1}(\bm\alpha(\bm\lambda^*,\bm\eta^*),\rho(\bm\lambda^*,\bm\eta^*))}{G_C(\bm\alpha(\bm\lambda^*,\bm\eta^*),\rho(\bm\lambda^*,\bm\eta^*))}
=\frac{v_{kl}(\tilde{\bm\lambda},\tilde{\bm\eta})}{v_{kl}(\bm\lambda^*,\bm\eta^*)}
=\frac{u_{pq}(\tilde{\bm\lambda},\tilde{\bm\eta})}{u_{pq}(\bm\lambda^*,\bm\eta^*)}, \ \forall i\in[N], (k,l)\in\A, p,q \in [N].
\end{align*}
\end{lemma}

Finally, under the optimal dynamic policy $(\bm\lambda^*,\bm\eta^*)$, since the expected number of idle vehicles at location $i$ is at least $\alpha_i(\bm\lambda^*,\bm\eta^*)$, while the expected numbers of vehicles traveling on trip $(i,j)$ and on redirection $(i,j)$ are $v_{ij}(\bm\lambda^*,\bm\eta^*)/\mu^{ij}$ and $u_{ij}(\bm\lambda^*,\bm\eta^*)/\nu^{ij}$, Little's law implies that $\sum_{i=1}^N\alpha_i(\bm\lambda^*,\bm\eta^*)+\rho(\bm\lambda^*,\bm\eta^*)\leq C$. Thus, $(\bm\alpha(\bm\lambda^*,\bm\eta^*),\rho(\bm\lambda^*,\bm\eta^*))$ is feasible for the same problem \eqref{min: ratio}. Lemma~\ref{lem: solution} and \eqref{eq: gamma} then show that $  \frac{\R(\tilde{\bm\lambda}, \tilde{\bm\eta})}{\R( \bm \lambda^*, \bm \eta^* )}  \geq  \frac{G_{C-1}(\bm\alpha(\bm\lambda^*, \bm\eta^*),\rho(\bm\lambda^*,\bm\eta^*))}{G_C(\bm\alpha(\bm\lambda^*, \bm\eta^*),\rho(\bm\lambda^*, \bm\eta^*))}\geq\gamma_{C,N}$, which completes the proof. \hfill \Halmos
\end{proof}

Note that the optimal static policy can be computed by solving the following convex optimization problem. Since its derivation follows the same steps as in Section~\ref{sec:opt-static}, we omit the proof.
\begin{equation}\label{opt:convex-program-redirection}
\begin{aligned}
\max_{\bm\omega\geq\bm 0,\bm z,\bm q\geq\bm 0}\quad
&\sum_{(i,j)\in\A}\omega_i r^{ij}\left(\frac{z_{ij}}{\omega_i}\right)-\sum_{i,j\in[N]}c_{ij}q_{ij}\\
\mathrm{s.t.}\quad
&\sum_{j:(i,j)\in\A}z_{ij}+\sum_{j\in[N]}q_{ij}=\sum_{k:(k,i)\in\A}z_{ki}+\sum_{k\in[N]}q_{ki} &&i\in[N],\\
&\sum_{j\in[N]}q_{ij}\leq\sum_{k:(k,i)\in\A}z_{ki} &&i\in[N],\\
&0\leq z_{ij}\leq\Lambda_{ij}\omega_i &&(i,j)\in\A,\\
&Q_C(\bm\omega,\bm z,\bm q)\leq1,
\end{aligned}
\end{equation}
where $Q_C(\bm\omega, \bm z, \bm q):= \frac{G_C\left(\bm\omega, \sum_{(i,j)\in\A} z_{ij}/\mu^{ij} + \sum_{i,j\in[N]} q_{ij}/\nu^{ij} \right)}{G_{C-1}\left(\bm\omega,\sum_{(i,j)\in\A} z_{ij}/\mu^{ij} + \sum_{i,j\in[N]} q_{ij}/\nu^{ij} \right)}  $. For an optimal solution $(\bm\omega^*, \bm z^*, \bm q^* )$ of \eqref{opt:convex-program-redirection}, an optimal static policy is $ \lambda^{ij} = z_{ij}^* /\omega_i^*$ for $(i,j) \in \A $ and $\eta^{ij}=\frac{q_{ij}^*}{\sum_{k:(k,i)\in\A}z_{ki}^*} $ for $ i,j \in [N] $.



\begin{APPENDICES}

\section{Omitted Proofs}\label{app:proofs}

\begin{proof}{\textbf{Proof of Lemma~\ref{lem:little's law}}.}
For each $(i,j)\in\A$, vehicles start traveling on trip $(i,j)$ at the long-run average rate $v_{ij}(\bm\lambda^*)$, and each such trip has mean travel time $1/\mu^{ij}$. Therefore, applying Little's law to the vehicles traveling on $(i,j)$, the expected number of vehicles on this trip under $\bm\lambda^*$ is $ \sum_{\bm x\in\mathcal X_{C,N}}x_{ij}\mathbb P_{\bm x}(\bm\lambda^*)=\frac{v_{ij}(\bm\lambda^*)}{\mu^{ij}} $. At location $i$, the expected number of idle vehicles is $\sum_{\bm x\in\mathcal X_{C,N}}x_i \mathbb P_{\bm x}(\bm\lambda^*)$. Since $x_i\geq1$ whenever location $i$ is available, $  \sum_{\bm x\in\mathcal X_{C,N}}x_i\mathbb P_{\bm x}(\bm\lambda^*)\geq\sum_{\bm x\in\mathcal X_{C,N}:x_i\geq1}\mathbb P_{\bm x}(\bm\lambda^*)=\alpha_i(\bm\lambda^*)$. Because the system contains exactly $C$ vehicles at all times, the expected numbers of idle and traveling vehicles sum to $C$. Hence,
\begin{align*} 
C =\sum_{i=1}^N\sum_{\bm x\in\mathcal X_{C,N}} x_i\mathbb P_{\bm x}(\bm\lambda^*) + \sum_{(i,j)\in\A}\sum_{\bm x\in\mathcal X_{C,N}}x_{ij}\mathbb P_{\bm x}(\bm\lambda^*) 
\geq\sum_{i=1}^N\alpha_i(\bm\lambda^*)+\sum_{(i,j)\in\A}\frac{v_{ij}(\bm\lambda^*)}{\mu^{ij}} = \sum_{i=1}^N \alpha_i(\bm\lambda^*) + \rho(\bm\lambda^*). \hfill \Halmos 
 \end{align*}
\end{proof}

\begin{proof}{\textbf{Proof of Lemma~\ref{lem:availability-ratio-property}}.}

We first prove in Lemma \ref{lem:ratio-monotonicity} that $\frac{G_{k-1}(\bm\alpha,\rho)}{G_k(\bm\alpha,\rho)}$ is non-decreasing in $k$.

\begin{lemma}\label{lem:ratio-monotonicity}
For every $\bm\alpha\in\mathbb R_+^N$ and $\rho\geq 0$, $G_{k-1}(\bm\alpha,\rho)/G_k(\bm\alpha,\rho)$ is non-decreasing in $k$. 
\end{lemma}
\begin{proof}{Proof.}
Let $G_k^{(0)}=\rho^k/k!$ and recursively $G_k^{(i)}=G_k^{(i-1)}+\alpha_iG_{k-1}^{(i)}$ for $i=1,\ldots,N$, so that $G_k^{(N)}=G_k(\bm\alpha,\rho)$. We show by induction on $i$ that $G_{k-1}^{(i)}/G_k^{(i)}$ is non-decreasing in $k$. For $i=0$, $G_{k-1}^{(0)}/G_k^{(0)}=k/\rho$, which is increasing in $k$. Suppose the claim holds for $i-1$. We observe that
\begin{align*}
\left(G_k^{(i)}\right)^2-G_{k-1}^{(i)}G_{k+1}^{(i)}
&= G_k^{(i)}  G_k^{(i-1)} - G_{k-1}^{(i)} G_{k+1}^{(i-1)}\\
&=\alpha_i^kG_0^{(i-1)}G_k^{(i-1)}+\sum_{l=0}^{k-1}\alpha_i^l\left(G_k^{(i-1)}G_{k-l}^{(i-1)}-G_{k+1}^{(i-1)}G_{k-1-l}^{(i-1)}\right)\geq0.
\end{align*}
The first equality adds and subtracts the common term $\alpha_iG_{k-1}^{(i)}G_k^{(i)}$ and uses the recursion. The second equality follows from $G_k^{(i)}=\sum_{l=0}^k\alpha_i^lG_{k-l}^{(i-1)}$ and $G_{k-1}^{(i)}=\sum_{l=0}^{k-1}\alpha_i^lG_{k-1-l}^{(i-1)}$. The inequality follows from the  induction hypothesis: $G_{k-l}^{(i-1)}/G_{k-1-l}^{(i-1)}\geq G_{k+1}^{(i-1)}/G_k^{(i-1)}$. Thus, $G_{k-1}^{(i)}/G_k^{(i)}\leq G_k^{(i)}/G_{k+1}^{(i)}$, proving the induction. Taking $i=N$ proves that $G_{k-1}(\bm\alpha, \rho)/G_k(\bm\alpha,\rho)$ is non-decreasing in $k$. The result extends to $\rho=0$ by continuity.  \hfill \Halmos 
\end{proof}


We now prove Lemma~\ref{lem:availability-ratio-property}. Let $G_{-1}(\alpha,\rho)=0 $. For any $k$, the partial derivatives of $ G_k(\bm\alpha, \rho)$  are given by 
\begin{align}\label{eq:basic-derivatives}
\frac{\partial G_k(\bm\alpha, \rho)}{\partial \rho }=G_{k-1}(\bm\alpha,\rho), \qquad \frac{\partial G_k(\bm\alpha, \rho)}{\partial \alpha_i}=\sum_{l=1}^k\alpha_i^{l-1}G_{k-l}(\bm\alpha, \rho ), \ i \in [N] .
\end{align}

(i) Recall that $ R(\bm\alpha, \rho)  = \frac{G_{C-1}(\bm\alpha, \rho)}{G_C(\bm\alpha, \rho)}  $. For any $ \delta > 0 $, direct computation shows that
\begin{align*}
R(\bm\alpha,\rho+\delta)-R(\bm\alpha,\rho)=\int_{\rho}^{\rho+\delta}\frac{\partial R(\bm\alpha,t)}{\partial t}\,dt=\int_{\rho}^{\rho+\delta}\frac{G_{C-2}(\bm\alpha,t) G_C(\bm\alpha,t) - (G_{C-1}(\bm\alpha,t))^2}{(G_C(\bm\alpha,t))^2}\,dt\leq0,
\end{align*}
where the inequality follows from Lemma~\ref{lem:ratio-monotonicity} that $G_{C-2}/G_{C-1}\leq G_{C-1}/G_C$. 

(ii) For any $ i \in [N] $ and  $ 0 \leq \delta \leq \rho $, applying \eqref{eq:basic-derivatives} implies that 
\begin{align*}
&R(\bm\alpha+\delta\bm e_i,\rho-\delta)-R(\bm\alpha,\rho) \\
=&\int_0^\delta \frac{\partial R(\bm\alpha+t\bm e_i,\rho-t)}{\partial\alpha_i}-\frac{\partial R(\bm\alpha+t\bm e_i,\rho-t)}{\partial\rho} \; dt\\
=& \int_0^\delta \frac{\sum_{l=2}^C(\alpha_i + t)^{l-1} \left[  G_C(\bm\alpha + t\bm e_i,\rho-t) G_{C-1-l}(\bm\alpha + t\bm e_i,\rho-t) - G_{C-l}(\bm\alpha + t\bm e_i,\rho-t)G_{C-1}(\bm\alpha+t\bm e_i,\rho-t) \right]}{ (G_C(\bm\alpha + t\bm e_i, \rho-t))^2}dt\\
\leq & 0,
\end{align*}
where the last inequality follows from Lemma~\ref{lem:ratio-monotonicity} that $G_{C-l - 1}/G_{C- l  }\leq G_{C-1}/G_C$.

(iii) For any $i,j\in[N]$ with $\alpha_i\geq\alpha_j$ and any $0\leq\delta\leq\alpha_j$, applying \eqref{eq:basic-derivatives} implies that 
\begin{align*}
&R(\bm\alpha+\delta\bm e_i-\delta\bm e_j,\rho)-R(\bm\alpha,\rho)\\
=&\int_0^\delta\left(\frac{\partial R(\bm\alpha+t\bm e_i-t\bm e_j,\rho)}{\partial\alpha_i}-\frac{\partial R(\bm\alpha+t\bm e_i-t\bm e_j,\rho)}{\partial\alpha_j}\right)dt\\
=&\int_0^\delta \frac{\sum_{l=2}^C \left( (\alpha_i+t)^{l-1}-(\alpha_j-t)^{l-1} \right)  G_C(\bm\alpha + t\bm e_i - t\bm e_j,\rho) G_{C-1-l}(\bm\alpha + t\bm e_i - t\bm e_j, \rho  ) }{( G_C(\bm\alpha + t\bm e_i - t\bm e_j,\rho) )^2 } \; dt\\
& -  \int_0^\delta \frac{\sum_{l=2}^C \left( (\alpha_i+t)^{l-1}-(\alpha_j-t)^{l-1} \right)   G_{C-l}(\bm\alpha + t\bm e_i - t\bm e_j,\rho) G_{C-1}(\bm\alpha + t\bm e_i - t\bm e_j,\rho) }{( G_C(\bm\alpha + t\bm e_i - t\bm e_j,\rho) )^2 } \; dt  \leq 0,
\end{align*}
where the inequality follows from $ \alpha_i \geq \alpha_j $ and  Lemma~\ref{lem:ratio-monotonicity} that $G_{C-l - 1}/G_{C- l  }\leq G_{C-1}/G_C$.  
\hfill \Halmos  \\

\end{proof}

\begin{proof}{\textbf{Proof of Lemma~\ref{lem: solution}}.}
Start from any feasible $(\bm\alpha,\rho)$ of \eqref{min: ratio}. If $\sum_{i=1}^N\alpha_i+\rho<C$, increase $\rho$ until the constraint becomes binding. By Lemma~\ref{lem:availability-ratio-property}(i), this cannot increase $R(\bm\alpha,\rho)$. Thus, we restrict attention to solutions satisfying $\sum_{i=1}^N\alpha_i+\rho=C$. Second, suppose $0<\alpha_j\leq\alpha_i<1$ for some $i,j\in[N]$. Increase $\alpha_i$ and decrease $\alpha_j$ by the same amount until either $\alpha_i=1$ or $\alpha_j=0$. By Lemma~\ref{lem:availability-ratio-property}(iii), this operation cannot increase $R(\bm\alpha,\rho)$. Repeating this operation leaves at most one $\alpha_i$ strictly between zero and one. Finally, as long as $\rho>0$ and some $\alpha_i<1$, decrease $\rho$ and increase $\alpha_i$ by the same amount until either $\rho=0$ or $\alpha_i=1$. By Lemma~\ref{lem:availability-ratio-property}(ii), this operation cannot increase $R(\bm\alpha,\rho)$. Apply this operation first to the possible fractional $\alpha_i$ and then to those equal to zero. If $\rho=0$, then $\sum_{i=1}^N\alpha_i=C$. Since at most one $\alpha_i$ is fractional and the sum is an integer, every $\alpha_i$ must belong to $\{0,1\}$, with exactly $C$ of them equal to one. This requires $N\geq C$. If instead all $\alpha_i$ reach one while $\rho>0$, then $\rho=C-N>0$, so $N<C$. Therefore, in either case, there exists an optimal solution satisfying $\rho=C-\min\{N,C\}$, $\sum_{i=1}^N\alpha_i=\min\{N,C\}$, and $\alpha_i\in\{0,1\}$ for all $i\in[N]$.
\hfill \Halmos \\

\end{proof}

\begin{proof}{\textbf{Proof of Lemma~\ref{lem:tightness}}.} 
Consider a directed cycle $\A=\{(1,2),\ldots,(N-1,N),(N,1)\}$. For $(i,j)\in\A$, let $\Lambda_{ij}=1$ and $H_{ij}$ be exponential with mean $ \frac{1}{\epsilon} $. For simplicity, we define $m=\min\{N,C\}$ and  denote $Q_C(\bm\alpha, \rho)= \frac{1}{R(\bm\alpha,\rho)}= \frac{G_C(\bm\alpha, \rho)}{G_{C-1}(\bm\alpha,\rho)}$, $\hat{\bm\alpha}:=(\bm 1_m,\bm 0_{N-m})$, and $ \hat{\rho} = C - m $. Note that $ Q_C(\bm\alpha, \rho) $ is convex by Theorem~\ref{thm:convex-static-program} and $ Q_C(\hat{\bm\alpha}, \hat\rho) = 1/\gamma_{C,N}  $ from \eqref{eq: gamma}. We define $a:=\frac{\partial Q_C}{\partial\alpha_i}(\hat{\bm\alpha}, \hat\rho)$ for $i\leq m$ and $b:=\frac{\partial Q_C}{\partial\rho}(\hat{\bm\alpha},\hat\rho)$. By symmetry of $ Q_C(\bm \alpha, \rho) $, $a$ is the same for all $i\leq m$. Moreover, Lemma~\ref{lem:availability-ratio-property} implies $a \geq b \geq 0$.  Since $Q_C$ is homogeneous of degree one, differentiating $Q_C(t\hat{\bm\alpha}, t\hat\rho) = tQ_C(\hat{\bm\alpha}, \hat\rho)$ with respect to $t$ at $t=1$ yields $Q_C(\hat{\bm\alpha}, \hat\rho) = \sum_{i=1}^N\frac{\partial Q_C}{\partial\alpha_i}(\hat{\bm\alpha}, \hat\rho)\hat\alpha_i + \frac{\partial Q_C}{\partial\rho}(\hat{\bm\alpha}, \hat\rho)\hat\rho = ma + b\hat\rho =\frac{1}{\gamma_{C,N}}>0$, which further implies $ a > 0 $.  For the first $m$ arcs along the cycle, the revenue rate functions are $ r^{ij}(\lambda)= a + \frac{bN}{m\epsilon}\lambda$, which is induced by a shifted and scaled version of the limiting triangular distribution studied in \cite{jin2019tight} with $p^{ij}(\lambda)= \frac{b N}{m\epsilon} + \frac{a}{\lambda}$. For the remaining arcs, the revenue rates functions are $r^{ij}(\lambda)= \epsilon \lambda (1 - \lambda) $. Since $Q_C(\bm \alpha, \rho)$ is convex, we consider its tangent hyperplane at $(\hat{\bm\alpha},\hat\rho)$: 
\begin{align}
    Q_C(\bm\alpha,\rho)\geq 
      Q_C(\hat{\bm\alpha},\hat\rho)+ a\sum_{i=1}^m(\alpha_i-1) + \sum_{i=m+1}^N\frac{\partial Q_C}{\partial\alpha_i}(\hat{\bm\alpha},\hat\rho)\alpha_i + b(\rho-\hat\rho) 
      \geq& a\sum_{i=1}^m\alpha_i+b\rho. \label{eq: a-b-lower}
\end{align}
The inequality follows from $ Q_C(\hat{\bm\alpha}, \hat\rho) = ma + b \hat{\rho} $ and $\frac{\partial Q_C}{\partial\alpha_i}(\hat{\bm\alpha},\hat\rho)\geq 0$ by Lemma~\ref{lem:availability-ratio-property}.

Now, consider any static policy $\bm\lambda$. Following the formulation \eqref{opt:convex-program} in Section~\ref{sec:opt-static}, let $\bm\omega$ and $\bm z$ be the corresponding variables. On the directed cycle, the constraint $\sum_{j:(i,j)\in\A}z_{ij}=\sum_{k:(k,i)\in\A}z_{ki}$ implies that all $z_{ij}$ are equal, denoted by $z$. Hence, the total travel load $\rho = \sum_{(i,j)\in\A} z_{ij}/ \epsilon=Nz/\epsilon$. Then, the revenue contribution from the first $m$ arcs is $ \frac{\sum_{i=1}^m\omega_i\left(a+\frac{bN}{m\epsilon}\frac{z}{\omega_i}\right)}{Q_C(\bm\omega,\rho)}= \frac{a\sum_{i=1}^m \omega_i + b\rho}{Q_C(\bm\omega,\rho)}
\leq1 $, where the inequality follows from \eqref{eq: a-b-lower} that $Q_C(\bm\omega,\rho)\geq a\sum_{i=1}^m\omega_i+b\rho$. For each of the remaining $N-m$ arcs, $r^{ij}(\lambda) \leq \frac{\epsilon}{4}$. We conclude that
\begin{align*}
\R^{sta}(\bm\lambda)\leq 1 + \frac{(N-m)\epsilon}{4} \Longrightarrow \R^{sta*}\leq1+\frac{(N-m)\epsilon}{4}.
\end{align*}

We next construct a dynamic policy whose revenue converges to $1/\gamma_{C,N}$ when $ \epsilon \to 0 $. Consider the dynamic policy  $\hat{\lambda}_{\bm s}^{ij}=\epsilon^2$ if $i\leq m$ and $x_i=1$, and $\hat{\lambda}_{\bm s}^{ij}=1$ otherwise whenever $x_i\geq1$. Since this policy releases vehicles slowly when only one vehicle remains, and quickly when there are multiple vehicles, the first $m$ locations each tend to have exactly one vehicle. Note that on the directed cycle, every arc has the same throughput, denoted by $ \hat{v}$. Thus, the total travel load is $\rho(\hat{\bm\lambda}) = N \hat{v} /\epsilon$. Also, the last $N-m$ arcs generate zero revenue under $\hat{\bm\lambda}$ since $r^{ij}(1)=0$. The revenue under $ \hat{\bm\lambda} $ is given by
\begin{align*}
  \R (\hat{\bm\lambda} ) = \int_{\bm s \in \mathcal{S}_{C,N}}  \sum_{i=1 }^m \mathbf{1}\{x_i\geq1\} (a + \frac{b N}{m \epsilon} \hat{\lambda}^{i,i+1}_{\bm s} ) \; d \Pi^{\hat{\bm \lambda}}( \bm s) = a \sum_{i=1}^m \alpha_i(\hat{\bm\lambda}) + \frac{bN}{m\epsilon}\sum_{i=1}^m \hat{v} = a \sum_{i=1}^m \alpha_i(\hat{\bm\lambda}) + b \rho(\hat{\bm \lambda}).
\end{align*}
Since the effective arrival rate at each location depends only on the number of vehicles at that location, the induced system has the queue-dependent product form of \citet{reiser1975queuing}. The steady-state probability of each count state $\bm x\in\mathcal X_{C,N}$ can be written as
\begin{align*}
\mathbb P_{\bm x} =\frac{\prod_{i=1}^m \phi_\epsilon(x_i)\prod_{i=m+1}^N\epsilon^{x_i}\prod_{(i,j)\in\A}1/x_{ij}!}{\sum_{\bm x'\in\mathcal X_{C,N}}\prod_{i=1}^m\phi_\epsilon(x_i')\prod_{i=m+1}^N\epsilon^{x_i'}\prod_{(i,j)\in\A}1/x_{ij}'!},\qquad \phi_\epsilon(0)=1,\quad \phi_\epsilon(x)=\epsilon^{x-2}\ \text{for } x\geq1.
\end{align*}
Let $ \mathcal{X}_{C,N}(m) $ be the subset of $ \mathcal{X}_{C,N} $ satisfying $x_i=1$ for $i \in [m]$ and $x_i=0$ for $i=m+1,\ldots,N$ with $\sum_{(i,j)\in\A}x_{ij}=C-m$. We observe that $ \lim_{\epsilon \to 0 } \mathbb{P}_{\bm x} > 0 $ for $ \bm{x} \in \mathcal{X}_{C,N}(m) $ and $ \lim_{\epsilon \to 0 } \mathbb{P}_{\bm x} = 0 $ otherwise. To see this, for any $\bm x\in \mathcal X_{C,N}$, collecting the factors of $\epsilon$ in the numerator of $\mathbb P_{\bm x}$ yields that $ \prod_{i=1}^m \phi_\epsilon(x_i) \prod_{i=m+1}^N \epsilon^{x_i}
=\epsilon^{\sum_{i=1}^m (x_i - 2) \mathbf{1} \{x_i \geq 1\} + \sum_{i=m+1}^N x_i} =\epsilon^{\sum_{i=1}^N x_i-2\sum_{i=1}^m\mathbf{1}\{x_i\geq1\}} \leq \epsilon^{-m} $, where the inequality holds with equality for $ \bm x \in \mathcal{X}_{C,N}(m) $. Then, for $i \in [m] $, $ 1 \geq \lim_{\epsilon \to 0} \alpha_i(\hat{\bm \lambda}) \geq \lim_{\epsilon \to 0} \sum_{\bm x \in \mathcal{X}_{C,N}(m)  }\mathbb P_{\bm x} =1 $. Moreover, since count states in $ \mathcal{X}_{C,N}(m) $ have  $C-m$ vehicles traveling, Little's law implies $ \lim_{\epsilon \to 0} \rho(\hat{\bm \lambda}) = C-m$. Thus, we conclude that 
\begin{align*}
\gamma_{C,N}
\leq \liminf_{\epsilon\to0}\frac{\R^{sta*}}{\R^*}
\leq \limsup_{\epsilon\to0}\frac{1+(N-m)\epsilon/4}{\R^*} 
\leq \lim_{\epsilon\to0}\frac{1+(N-m)\epsilon/4}{\R(\hat{\bm\lambda})} 
=\gamma_{C,N}.
\end{align*}
The first inequality holds by the lower bound in Theorem~\ref{thm:main-guarantee}. The second inequality follows from $\R^{sta*} \leq 1 + \frac{(N-m)\epsilon}{4}$. The last equality holds by $  \lim\limits_{\epsilon \to 0 }  \R(\hat{\bm\lambda}) =\lim_{\epsilon \to 0 } a \sum_{i=1}^m \alpha_i(\hat{\bm\lambda}) + b \rho(\hat{\bm \lambda}) = ma + b(C-m) = \frac{1}{\gamma_{C,N}}  $.  \hfill \Halmos \\

\end{proof}

\begin{proof}{\textbf{Proof of Lemma~\ref{lem: lowerbound-gamma}}.} We first show that $ \gamma_{C,N} \geq C/(2C -1) $. Suppose that $N \geq C $. Lemma~\ref{lem: solution} shows that $\hat{\rho}=0$, which implies that $ \gamma_{C,N} = R(\hat{\bm\alpha},0) =  \binom{2C-2}{C-1} / \binom{2C-1}{C} = \frac{C}{2C-1} $ as only the $k=0$ terms are non-zero. Suppose now that $N<C$. Lemma~\ref{lem: solution} implies that $\hat{\alpha}_i=1$ for all $i\in[N]$ and $\hat{\rho}=C-N$. Applying \eqref{eq:G_C}, we have
\begin{align*}
C \cdot G_C(\hat{\bm\alpha},\hat{\rho})
&=\sum_{k=0}^{C}\frac{(C-N)^k}{k!}\big((C-k)+k\big)\binom{N+C-k-1}{C-k} \\
& = \sum_{k=0}^{C-1}\frac{(C-N)^k}{k!} (N + C -k -1) \binom{N+C-k-2}{C-k-1} + \sum_{k=1}^{C} \frac{(C-N)^{k}}{(k-1)!} \binom{N + C -k -1}{C - k}     \\
&=(2C-1)G_{C-1}(\hat{\bm\alpha},\hat{\rho}) - \sum_{k=1}^{C-1} \frac{(C-N)^k}{(k-1)!} \binom{N+C-k-2}{C-k-1}
 \leq (2C-1) G_{C-1}(\hat{\bm\alpha},\hat{\rho}).
\end{align*}
The second equality follows from $(C-k)\binom{N+C-k-1}{C-k}=(N+C-k-1)\binom{N + C-k-2}{C-k-1}$ for $ k \leq C-1 $. The third equality follows from  $N + C - k - 1 = 2C -1 - (C-N) - k $ and canceling out identical terms. Therefore, $ \gamma_{C,N} = \frac{G_{C-1}(\hat{\bm\alpha},\hat{\rho})}{G_C(\hat{\bm\alpha}, \hat{\rho})}\geq\frac{C}{2C-1} $.  

Next, we prove that $ \gamma_{C,N}\geq \frac{\sqrt{C}}{\sqrt{C}+\sqrt{\min\{N,C\}}} $. According to \eqref{eq:G_C} and Lemma~\ref{lem: solution} that $\hat{\rho}=C-\min\{N,C\}$, $\sum_{i=1}^N \hat{\alpha}_i = \min\{N,C\}  $, $\hat{\alpha}_i \in \{0,1\}$, we have that $  G_0(\hat{\bm\alpha},\hat{\rho}) :=1 $, and $ G_i(\hat{\bm\alpha}, \hat{\rho}) = \sum_{k=0}^{i}\frac{\hat{\rho}^k}{k!} \binom{\min\{N,C\}+i-k-1}{i-k}, \ i \in [C] $. We further note that
\begin{align}
C \cdot G_C(\hat{\bm\alpha},\hat{\rho})
=&\sum_{k=0}^{C}\frac{\hat{\rho}^k}{k!}(C-k) \binom{\min\{N,C\}+C-k-1}{C-k} + \sum_{k=0}^{C}\frac{\hat{\rho}^k}{k!}k \binom{\min\{N,C\}+C-k-1}{C-k} \notag \\
=& \min\{N,C\} \sum_{k=0}^{C-1}\frac{\hat{\rho}^k}{k!}\binom{\min\{N,C\}+C-k-1}{C-k-1} + \sum_{k=1}^C \frac{\hat{\rho}^k}{(k-1)!} \binom{\min\{N,C\}+C-k-1}{C-k} \notag \\
=& \min\{N,C\}  \sum_{k=0}^{C-1}\frac{\hat{\rho}^k}{k!} \sum_{i=k}^{C-1}\binom{\min\{N,C\} + i-k-1}{i-k}  + \hat{\rho} \sum_{k=0}^{C-1} \frac{\hat{\rho}^k}{k!} \binom{\min\{N,C\}+C-k-2}{C- 1 - k}  \notag \\
=& \min\{N,C\}\sum_{i=0}^{C-1}G_i(\hat{\bm\alpha},\hat{\rho}) +  \hat{\rho}G_{C-1}(\hat{\bm\alpha},\hat{\rho}).  \label{eq: identity}
\end{align}
The first equality follows from $C=(C-k) + k$. The second equality follows from $ (C-k)\binom{\min\{N,C\} + C-k-1}{C-k}
=\min\{N,C\}\binom{\min\{N,C\} + C-k-1}{C-k-1}$. The third equality follows from the Hockey-stick identity $\binom{\min\{N,C\}+C-k-1}{C-k-1} = \sum_{i=k}^{C-1}\binom{\min\{N,C\}+ i - k - 1}{i-k}$. The last equality follows from the expressions of $G_i( \hat{\bm \alpha}, \hat{\rho} )$. By the monotonicity property in Lemma~\ref{lem:ratio-monotonicity}, for $i = 0,\ldots,C-1$, $ \frac{G_i(\hat{\bm\alpha}, \hat{\rho})}{G_C(\hat{\bm\alpha}, \hat{\rho})}
=\prod_{k=i+1}^{C} \frac{G_{k-1}(\hat{\bm\alpha}, \hat{\rho})}{G_k(\hat{\bm\alpha}, \hat{\rho})}
\leq \left(\frac{G_{C-1}(\hat{\bm\alpha},\hat{\rho})}{G_C(\hat{\bm\alpha},\hat{\rho})}\right)^{C-i}
=\gamma_{C,N}^{\,C-i} $. Then, applying this inequality and  dividing \eqref{eq: identity} by $G_C(\hat{\bm\alpha}, \hat{\rho})$  shows that
\begin{align*}
C \leq \min\{N,C\}\sum_{i=0}^{C-1}\gamma_{C,N}^{\,C-i} + (C-\min\{N,C\})\gamma_{C,N}  \leq \frac{\min\{N,C\}\gamma_{C,N}}{1-\gamma_{C,N}} + (C-\min\{N,C\})\gamma_{C,N}.
\end{align*}
Rearranging terms yields that $ (C-\min\{N,C\})\gamma_{C,N}^2 - 2C \gamma_{C,N} + C\leq0 $. Since $\gamma_{C,N}<1$, solving the quadratic inequality gives $ \gamma_{C,N}\geq\frac{\sqrt{C}}{\sqrt{C}+\sqrt{\min\{N,C\}}} $. Together, we conclude that $ \gamma_{C,N} \geq \frac{C}{C+\min\{C-1,\sqrt{CN}\}} = \max\{ \frac{C}{2C-1}, \frac{\sqrt C}{\sqrt{C} +   \sqrt{\min\{N,C\} } } \}  $.  \hfill \Halmos \\
\end{proof}

\begin{proof}{\textbf{Proof of Corollary~\ref{prop: dominate}}.}

By Lemma \ref{lem: lowerbound-gamma}, it suffices to show that $ \frac{\sqrt{C}}{\sqrt{C}+\sqrt{\min\{N,C\}}} > B_{C,N} = 1 + \frac{N-1}{2C}+\frac{1}{2(N+C)}-\sqrt{\left(\frac{N-1}{2C}+\frac{1}{2(N+C)}\right)^2+\frac{N-1}{C}+\frac{1}{N+C}}  $.  For any $x\geq0$ and $s\in(0,1]$, we note that  $ 1+\frac{x}{2}-\sqrt{\frac{x^2}{4}+x} < s   \Leftrightarrow  x > \frac{(1-s)^2}{s}$. $B_{C,N}$ is obtained by setting $ x=\frac{N-1}{C}+\frac{1}{N+C} = \frac{N(N+C-1)}{C(N+C)}$.  If $N < C$, then $ (1-\frac{\sqrt{C}}{\sqrt{C}+\sqrt{N}})^2/\frac{\sqrt{C}}{\sqrt{C}+\sqrt{N}} = \frac{N}{C+\sqrt{CN}} <  \frac{N(N+C-1)}{C(N+C)} $ as  $(N+C-1)\sqrt{CN} > C$. If $N\geq C$, then $ \sqrt{C}/(\sqrt{C}+\sqrt{\min\{N,C\}}) = \frac{1}{2} < \frac{N(N+C-1)}{C(N+C)} $. Therefore, we conclude that $ \gamma_{C,N} \geq \frac{\sqrt{C}}{\sqrt{C} + \sqrt{\min\{N,C\}}}
> B_{C,N}$.  \hfill \Halmos \\
\end{proof}


\begin{proof}{\textbf{Proof of Theorem~\ref{thm:zero-travel}}.} 

We first prove that $ \inf_{\Omega^0_{C,N}} \frac{\R^{\tilde{\bm\lambda}}}{\R^*} \geq \frac{C}{C+\min\{C,N\}-1} $. Since no vehicle is in transit when all travel times are zero, $\rho(\bm\lambda^*)=0$ in this case. Lemma~\ref{lem: RevRatio_lower} thus yields that $ \frac{\R^{\tilde{\bm\lambda}}}{\R^*} \geq R(\bm{\alpha}(\bm \lambda^*), 0) $. 

Similar to Lemma~\ref{lem: solution}, we claim that there exists an optimal solution to $ \min_{\bm\alpha\in[0,1]^N, \ \sum_{i=1}^N \alpha_i \leq C } \R(\bm\alpha,0) $  with $  \sum_{i=1}^N\hat\alpha_i = \min\{N,C\} $ and $ \hat\alpha_i \in\{0,1\} $ for all $i \in [N] $. To see this, we first note that increasing any $ \alpha_i $ cannot increase $R(\bm\alpha,0)$. This is because for any feasible $\delta\geq0$, Lemma~\ref{lem:availability-ratio-property}(ii) and (i) imply that $ R(\bm\alpha + \delta\bm e_i,0) \leq R(\bm\alpha,\delta) \leq R(\bm\alpha,0) $. Then, starting from any feasible $\bm\alpha$, we can increase $\alpha_i$ until either $\sum_{i=1}^N\alpha_i=C$ or $\alpha_i=1$ for all $i\in[N]$. Therefore, we restrict attention to solutions satisfying $\sum_{i=1}^N\alpha_i=\min\{N,C\}$. Next, suppose $0<\alpha_j\leq\alpha_i<1$ for some $i,j\in[N]$. Increase $\alpha_i$ and decrease $\alpha_j$ by the same amount until either $\alpha_i=1$ or $\alpha_j=0$. By Lemma~\ref{lem:availability-ratio-property}(iii), this operation cannot increase $R(\bm\alpha,0)$. Repeating this operation leaves at most one $\alpha_i$ strictly between zero and one. Since $\sum_{i=1}^N\alpha_i=\min\{N,C\}$ is an integer, the remaining $\alpha_i$ cannot be fractional. Thus, there exists an optimal solution with $\sum_{i=1}^N\hat\alpha_i=\min\{N,C\}$ and $\hat\alpha_i\in\{0,1\}$ for all $i\in[N]$.
 
Combining the above observation with the proof of Theorem~\ref{thm:main-guarantee}, we conclude that 
\begin{align*}
 \frac{\R^{sta*}}{\R^*} \geq    \frac{\R^{\tilde{\bm\lambda}}}{\R^{*}} \geq   R(\bm{\alpha}(\bm \lambda^*), 0  ) \geq  R(\bm{\hat \alpha}, 0  ) 
     =  \frac{  \binom{\min\{N,C\}+ C-2}{C-1}}{\binom{\min\{N,C\}+C -1}{C}}   
     = \frac{C}{C +  \min\{ C, N \} - 1 },
\end{align*}
where the second equality follows from the expression of $ R(\bm{\hat{\alpha}}, 0 ) $, \eqref{eq:G_C} and weak compositions.  

For the tightness result, we provide an example with the matching lower bound, i.e., $ \frac{ \R^{sta*}}{\R^*} =  \frac{C}{C + \min\{C,N\}  -1} $. We consider the same directed cycle as in Lemma~\ref{lem:tightness} with $ \Lambda_{ij} =1 $ and zero travel times for all trips. Recall that $m=\min\{C,N\}$ and $Q_C(\bm\alpha, \rho)= 1/R(\bm\alpha,\rho)=G_C(\bm\alpha, \rho)/G_{C-1}(\bm\alpha,\rho)$. In this case, $\hat{\bm\alpha}=(\bm 1_m, \bm 0_{N-m})$ and $ \hat{\rho}=0 $. We define $a:=\frac{\partial Q_C}{\partial\alpha_i}(\hat{\bm\alpha},0)$ for $i\in[m]$. For the first $m$ arcs along the cycle, let $r^{ij}(\lambda)=a$, and for the remaining arcs let $r^{ij}(\lambda)= \delta\lambda(1-\lambda)$. We also have $ \frac{C + \min\{C,N\} -1 }{C} = \frac{C + m -1 }{C}  = Q_C(\hat{\bm\alpha},0) =  Q_C(\hat{\bm\alpha}, \hat\rho) = m a  $. Since $  \rho \equiv  0  $ with zero travel times, the same supporting-hyperplane argument as in Lemma~\ref{lem:tightness} yields that $ Q_C(\bm\alpha, 0) \ge a \sum_{i=1}^m\alpha_i$. Then, for any static policy $\bm\lambda$, we have that  $ R^{sta}(\bm\lambda) \leq \frac{a\sum_{i=1}^m \omega_i}{Q_C(\bm\omega, 0)} + \frac{(N-m) \delta}{4} \leq 1 + \frac{(N-m) \delta}{4}   $, and hence $\R^{sta*}\leq1+\frac{(N-m)\delta}{4}$.

Now consider the same dynamic policy as in Lemma 5,  reparameterized by $\delta$: $\hat{\lambda}_{\bm s}^{i,i+1}=\delta$ if $i\le m$ and $x_i=1$, and $\hat{\lambda}_{\bm s}^{i,i+1}=1$ otherwise whenever $x_i \ge 1$. The same queue-dependent product-form argument implies that, as $ \delta \to 0 $, all first $ m $ locations are nonempty with probability approaching one. The remaining $N-m$ arcs generate zero revenue under $\hat{\bm\lambda}$ since $r^{ij}(1)=0$. Therefore, $ \lim_{\delta \to 0} R(\hat{\bm \lambda}) = \lim_{\delta \to 0} a \sum_{i=1}^m \alpha_i (\hat{\bm \lambda}) = ma = \frac{C + m - 1}{C}  $. We conclude that $ \frac{C}{C+m-1}
\leq \liminf_{\delta\to0}\frac{\R^{sta*}}{\R^*}
\leq \limsup_{\delta\to0}\frac{1+(N-m)\delta/4}{\R^*}
\leq \lim_{\delta\to0} \frac{1+(N-m)\delta/4}{\R(\hat{\bm\lambda})}
=\frac{C}{C+m-1} $, which proves tightness. \hfill \Halmos \\
\end{proof}

\begin{proof}{\textbf{Proof of Lemma~\ref{lem:redirection-common-ratio}}.}
Under a static policy $(\bm\lambda,\bm\eta)$, the induced system is a closed BCMP network consisting of location nodes, trip nodes, and redirection nodes. Recall that $\lambda_i=\sum_{j:(i,j)\in\A}\lambda^{ij}$.  Let $w_i$, $w_{ij}$, and $q_{ij}$ denote a solution to the following traffic equations for these three types of nodes:
\begin{align*} 
w_{ij} = w_i \frac{\lambda^{ij}}{\lambda_i},\; (i,j)\in \A, \quad q_{ij}=\eta^{ij}\sum_{k:(k,i)\in\A }w_{ki},\; i,j\in[N],\quad w_i=\left(1-\sum_{j\in[N]} \eta^{ij}\right) \sum_{k:(k,i)\in\A}w_{ki}+\sum_{k\in[N]}q_{ki},\ i\in[N]. \end{align*}
We define $x_{ij}^r$ as the number of vehicles traveling on redirection $(i,j)$, the BCMP product form implies that the steady-state probability of a count state is proportional to
\begin{align*} \prod_{i=1}^N\left(\frac{w_i}{\lambda_i}\right)^{x_i}\prod_{(i,j)\in\A}\frac{1}{x_{ij}!}\left(\frac{w_{ij}}{\mu^{ij}}\right)^{x_{ij}}\prod_{i,j\in[N]}\frac{1}{x_{ij}^r!}\left(\frac{q_{ij}}{\nu^{ij}}\right)^{x_{ij}^r}. \end{align*}
We next identify a solution to these traffic equations under $(\tilde{\bm\lambda},\tilde{\bm\eta})$. Let $\tilde\lambda_i=\sum_{j:(i,j)\in\A}\tilde\lambda^{ij}$. From \eqref{eq:rebalancing-static-construction}, consider $w_i=\tilde\lambda_i \alpha_i(\bm\lambda^*, \bm\eta^*)$, $w_{ij}=v_{ij}(\bm\lambda^*, \bm\eta^*)$, and $q_{ij}=u_{ij}(\bm\lambda^*, \bm\eta^*)$. The first two traffic equations follow directly from
\begin{align*} 
w_i\frac{\tilde\lambda^{ij}}{\tilde\lambda_i}=\alpha_i(\bm\lambda^*,\bm\eta^*)\tilde\lambda^{ij}=v_{ij}(\bm\lambda^*,\bm\eta^*),\qquad \tilde\eta^{ij}\sum_{k:(k,i)\in\A}v_{ki}(\bm\lambda^*,\bm\eta^*)=u_{ij}(\bm\lambda^*,\bm\eta^*). \end{align*}
Moreover, stationarity of $(\bm\lambda^*,\bm\eta^*)$ implies  $\sum_{j:(i,j)\in\A}v_{ij}(\bm\lambda^*,\bm\eta^*)+\sum_{j\in[N]}u_{ij}(\bm\lambda^*,\bm\eta^*)=\sum_{k:(k,i)\in\A}v_{ki}(\bm\lambda^*,\bm\eta^*)+\sum_{k\in[N]}u_{ki}(\bm\lambda^*,\bm\eta^*)$. Hence, we obtain that
\begin{align*} 
\left(1-\sum_{j\in[N]}\tilde\eta^{ij}\right)\sum_{k:(k,i)\in\A}v_{ki}(\bm\lambda^*,\bm\eta^*) + \sum_{k\in[N]}u_{ki}(\bm\lambda^*,\bm\eta^*)
=\sum_{j:(i,j)\in\A}v_{ij}(\bm\lambda^*,\bm\eta^*)
=\tilde\lambda_i\alpha_i(\bm\lambda^*,\bm\eta^*), 
\end{align*}
which verifies the third traffic equation. Therefore, the product-form weights under $(\tilde{\bm\lambda}, \tilde{\bm\eta})$ are $\alpha_i(\bm\lambda^*, \bm\eta^*)$ for location $i$, $v_{ij}(\bm\lambda^*, \bm\eta^*)/\mu^{ij}$ for trip $(i,j)$, and $u_{ij}(\bm\lambda^*, \bm\eta^*)/\nu^{ij}$ for redirection $(i,j)$. Applying the same argument as in Lemma~\ref{lem: RevRatio_lower} and using the multinomial theorem to aggregate the trip and redirection nodes shows that
\begin{align*} 
\frac{\alpha_i(\tilde{\bm\lambda}, \tilde{\bm\eta})}{\alpha_i(\bm\lambda^*, \bm\eta^*)} = \frac{G_{C-1}(\bm\alpha(\bm\lambda^*, \bm\eta^*),\rho(\bm\lambda^*, \bm\eta^*))}{G_C(\bm\alpha(\bm\lambda^*, \bm\eta^*), \rho(\bm\lambda^*,\bm\eta^*))}, \quad i\in[N]. 
\end{align*}
Finally, using \eqref{eq:rebalancing-static-construction} and the identities that $ v_{ij}(\tilde{\bm\lambda}, \tilde{\bm\eta}) = \tilde\lambda^{ij} \alpha_i(\tilde{\bm\lambda},\tilde{\bm\eta}) $ and $ u_{ij}(\tilde{\bm\lambda}, \tilde{\bm\eta}) = \tilde\eta^{ij} \sum_{k:(k,i)\in\A} v_{ki}(\tilde{\bm\lambda}, \tilde{\bm\eta})  $,
\begin{align*} 
\frac{v_{ij}(\tilde{\bm\lambda},\tilde{\bm\eta})}{v_{ij}(\bm\lambda^*, \bm\eta^*)} = \frac{ \tilde{\lambda}^{ij} \alpha_i(\tilde{\bm\lambda}, \tilde{\bm\eta})}{ \tilde{\lambda}^{ij} \alpha_i(\bm\lambda^*, \bm\eta^*)}  =\frac{\alpha_i(\tilde{\bm\lambda}, \tilde{\bm\eta})}{\alpha_i(\bm\lambda^*, \bm\eta^*)},\quad 
\frac{u_{kl}(\tilde{\bm\lambda},\tilde{\bm\eta})}{u_{kl}(\bm\lambda^*,\bm\eta^*)} = \frac{\sum_{j:(j,k)\in\A}v_{jk}(\tilde{\bm\lambda}, \tilde{\bm\eta})}{\sum_{j:(j,k)\in\A}v_{jk}(\bm\lambda^*,\bm\eta^*)} = \frac{G_{C-1}(\bm\alpha(\bm\lambda^*, \bm\eta^*),\rho(\bm\lambda^*, \bm\eta^*))}{G_C(\bm\alpha(\bm\lambda^*,  \bm\eta^*),\rho(\bm\lambda^*, \bm\eta^*))},
\end{align*}
for the positive throughputs specified in the lemma, which completes the proof.
\hfill \Halmos
\end{proof}

\end{APPENDICES}

\newpage
\bibliographystyle{ormsv080}
\bibliography{Paper}

\end{document}